\documentclass[english,11pt]{smfart}

\usepackage{amsbsy}
\usepackage{amsmath,amsfonts,amssymb,amsthm,mathrsfs,mathtools}

\usepackage{bm}

\usepackage[a4paper,vmargin={3cm,3cm},hmargin={3.5cm,3.5cm}]{geometry}
\usepackage[font=sf, labelfont={sf,bf}, margin=1cm]{caption}
\usepackage{graphicx}
\usepackage{epsfig}
\usepackage{latexsym}
\usepackage{xcolor}
\usepackage{ae,aecompl}
\usepackage{soul,framed}
\usepackage{comment}

\usepackage{xcolor}
\usepackage[pdfpagemode=UseNone,bookmarksopen=false,colorlinks=true,urlcolor=blue,citecolor=blue,citebordercolor=blue,linkcolor=blue]{hyperref}
\usepackage{smfhyperref}
\usepackage[capitalize]{cleveref}

\usepackage{pstricks}
\usepackage{enumerate}
\usepackage{pgfplots}
\usepackage{tikz,animate,media9}						
\usepackage{tikz-3dplot}

\usepackage{todonotes}
\usepackage{pifont}
\usepackage{bm,marvosym}
\usepackage{algorithm}
\usepackage{algorithmic}

\usepackage{bbm}

\usepackage{tcolorbox}

\definecolor{aleacolor}{rgb}{0.16,0.59,0.78}

\hypersetup{
	breaklinks,
	colorlinks=true,
	linkcolor=aleacolor,
	urlcolor=aleacolor,
	citecolor=aleacolor}

\newcount\colveccount
\newcommand*\colvec[1]{
	\global\colveccount#1
	\begin{pmatrix}
		\colvecnext
	}
	\def\colvecnext#1{
		#1
		\global\advance\colveccount-1
		\ifnum\colveccount>0
		\\
		\expandafter\colvecnext
		\else
	\end{pmatrix}
	\fi
}

\newcommand{\ndN}{\mathbb{N}}
\newcommand{\ndZ}{\mathbb{Z}}

\newcommand{\ndR}{\mathbb{R}}

\renewcommand{\Pr}[1]{\mathbb{P}(#1)}

\newcommand{\Prb}[1]{\mathbb{P}\left(#1\right)}

\newcommand{\Exb}[1]{\mathbb{E}\left[#1\right]}

\newcommand{\one}{{\mathbbm{1}}}

\newcommand{\convdis}{\,{\buildrel d \over \longrightarrow}\,}
\newcommand{\convd}{\,{\buildrel d \over \longrightarrow}\,}

\newcommand{\convp}{\,{\buildrel p \over \longrightarrow}\,}

\newcommand{\eqdist}{\,{\buildrel d \over =}\,}

\newcommand{\Seq}{\textsc{SEQ}}

\newcommand{\n}{\ensuremath{n_1,n_2,n_3}}

\newcommand{\myvec}[1]{\ensuremath{\left(\begin{smallmatrix}#1\end{smallmatrix}\right)}}

\DeclareMathOperator{\rt}{rt}
\newcommand{\bin}{\ensuremath{\mathrm{Binomial}}}

\newcommand{\fallfak}[2]{\ensuremath{#1^{\underline{#2}}}}
\newcommand{\auffak}[2]{\ensuremath{#1^{\overline{#2}}}}
\newcommand{\Stir}[2]{\genfrac{ \{ }{ \} }{0pt}{}{#1}{#2}}

\newtheorem{theorem}{Theorem}[section]

\newtheorem{corollary}[theorem]{Corollary}
\newtheorem{proposition}[theorem]{Proposition}
\newtheorem{lemma}[theorem]{Lemma}

\newtheorem{remark}[theorem]{Remark}
\newtheorem{example}[theorem]{Example}

\numberwithin{equation}{section}

\keywords{Gibbs partitions; Composition schemes; Lattice paths}

\title{\textbf{Gibbs partitions and lattice paths}}
\date{\today}

\author{Niccolò Bosio}
\address[Niccolò Bosio]{Vienna University of Technology}
\email{niccolo.bosio@tuwien.ac.at}

\author{Markus Kuba}
\address[Markus Kuba]{University of Applied Sciences - Technikum Wien}
\email{kuba@technikum-wien.at}

\author{Benedikt Stufler}
\address[Benedikt Stufler]{Vienna University of Technology}
\email{benedikt.stufler at tuwien.ac.at}

\pgfplotsset{compat=1.18}

\begin{document}

\vspace {-0.5cm}

\begin{abstract}
This work is devoted to the analysis of a Gibbs partition model, also known as a composition scheme. We consider a natural new condition on the component weights. It leads to a new behavior for the total number of components. We discover a condensation phenomenon, producing a unique giant component comprising almost the entire mass. Additionally, we prove a point process limit describing the asymptotic size of the  non-maximal  components exhibiting a sublinear power-law growth. A particular motivation for our article stems from  applications, ranging from simple random walks in the cube, over lattice paths models in the plane, pairs of directed random walks, over to urn models and card guessing games.
\end{abstract}


\maketitle
\section{Introduction and main results}
A great many combinatorial structures consist of more basic building
blocks or components. This situation is omnipresent in various fields, such as combinatorics, probability theory, and statistical mechanics. It manifests itself in a vast variety of topics, amongst others, permutations, random walks, trees, graphs, mappings, parking functions and urn models. 

An important direction of research is the development of a general theory for studying the asymptotic component distribution, which can be applied to all the special cases. Successful lines of research are partition models satisfying a conditioning relation~\cite{MR2032426,MR2121024} and the Gibbs partition model~\cite{MR2245368}, which has witnessed many developments in recent years~\cite{MR2453776, Stufler2018,MR4132643,Stufler2020,Stufler2022}. 
We note that in the combinatorial literature~\cite{FlaSed,BFSS2001, BaKuWa,FlaSed} the term composition scheme is often used instead. Roughly speaking, the Gibbs partition is based on two weight sequences that bias the size and number of components.

In this work we define and study a novel regime for the Gibbs partition model, assuming that the underlying weight sequences satisfy certain growth assumptions, see~\eqref{eq:condw}
and~\eqref{eq:condv}. We uncover several new phenomena and limit laws governing the number and sizes of components. 

We present concrete applications to combinatorial models.
In particular, we study a random walk on the cube with unit steps $(-1,0,0)$, $(0,-1,0)$, $(0,0,-1)$, as illustrated in Figure~\ref{fi:cubewalk}. The starting point is the upper corner $(n,n,n)$ and end point is the origin $(0,0,0)$. We give a complete description of the number of returns to the space diagonal, as well as many other statistics.

The study of such simple random walks is a fundamental topic in combinatorics and probability theory. Amongst the best known results is the classical transition between recurrence and transience for the drunkard walk established by P\'olya~\cite{Polya1921}: in dimensions one and two the simple random walk is recurrent, but for dimension $d\ge 3$ the behavior changes and the random walk is transience; this is vividly summarized by Kakutani: "A drunk man will find his way home, but a drunk bird may get lost forever". Additionally, we refer the reader to Fill et al.~\cite{Fill2005} for an analysis of P\'olya's drunkard walk using analytic combinatorics, see also~\cite[Example VI.14]{FlaSed}.

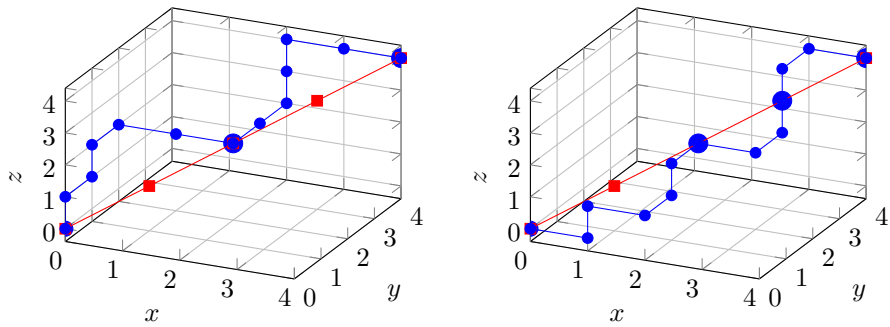
\begin{figure}[!htb]
	\resizebox{6cm}{6cm}
	{
		\begin{tikzpicture}
			\begin{axis}[
				grid=major, xlabel=$x$,ylabel=$y$,zlabel=$z$,
				xtick={0,...,4},  ytick={0,...,4}, ztick={0,...,4},
				] 
				\node [draw,circle,fill, blue,inner sep=2pt] at (axis cs: 0 , 0,0){}; 
				\node [draw,circle,fill, blue,inner sep=2.5pt] at (axis cs: 4 , 4,4){}; 
				\node [draw,circle,fill, blue,inner sep=2.5pt] at (axis cs: 2 , 2,2){}; 
				\addplot3[mark=square*, red] coordinates {(0,0,0) (1,1,1) (2,2,2)(3,3,3)(4,4,4)};
				\addplot3[mark=*, blue] coordinates {(4,4,4)(3,4,4)(2,4,4)(2,4,3)(2,4,2)(2,3,2)(2,2,2)(1,2,2)(0,2,2)(0,1,2)(0,1,1)(0,0,1)(0,0,0)};
			\end{axis}
		\end{tikzpicture} 
	}
	\quad
	\resizebox{6cm}{6cm}
	{
		\begin{tikzpicture}
			\begin{axis}[
				grid=major, xlabel=$x$,ylabel=$y$,zlabel=$z$,
				xtick={0,...,4},  ytick={0,...,4}, ztick={0,...,4},
				] 
				\node [draw,circle,fill, blue,inner sep=2pt] at (axis cs: 0 , 0,0){}; 
				\node [draw,circle,fill, blue,inner sep=2.5pt] at (axis cs: 4 , 4,4){}; 
				\node [draw,circle,fill, blue,inner sep=2.5pt] at (axis cs: 3 , 3,3){}; 
				\node [draw,circle,fill, blue,inner sep=2.5pt] at (axis cs: 2 , 2,2){}; 
				\addplot3[mark=square*, red] coordinates {(0,0,0) (1,1,1) (4,4,4)};
				\addplot3[mark=*, blue] coordinates {(4,4,4)(3,4,4)(3,3,4)(3,3,3)(3,3,2)(3,2,2)(2,2,2)(2,1,2)(2,1,1)(2,0,1)(1,0,1)(1,0,0)(0,0,0)};
			\end{axis}
		\end{tikzpicture} 
	}
	\caption{Two sample paths of length twelve from $(4,4,4)$ to $(0,0,0)$: the left path has two returns, the right paths has three returns to the diagonal.}
	\label{fi:cubewalk}
\end{figure}

 This generalizes the classical Dyck paths to dimension three, and we show a transition from $\sqrt{n}$ for dimension two to $\log(n)$ for the order of the returns to the diagonal. We also introduce and analyze Delannoy paths in the cube, generalizing a another classical model~\cite{Banderier_2005} of directed lattice paths to higher dimensions. Moreover, we discuss colored lattice paths in the cube, as well as pairs of directed lattice paths in two dimensions. Our results are also applied to certain lattice walks in the plane. Finally, applications of our main results to the sampling without replacement urn and to a card guessing game are stated.

\subsection{The Gibbs partition model}

Let $(w_n)_{n\ge 1}$ and $(v_k)_{k\ge 0}$ be two sequences of nonnegative real numbers, referred to as
\emph{component weights} and \emph{component-count weights}, respectively.
We denote their ordinary generating functions by
\[
W(z) = \sum_{n\ge 0} w_n z^n, 
\qquad
V(z) = \sum_{k\ge 1} v_k z^k.
\]
The associated \emph{composition scheme} is given by the formal composition
\[
U(z)=V\big(W(z)\big) = \sum_{n\ge 0} u_n z^n, 
\qquad
u_n = \sum_{k\ge 1} v_k w_n^{(k)}, \qquad w_n^{(k)} = \sum_{\substack{n_1+\dots+n_k = n \\ n_1, \ldots, n_k \ge 0}} \prod_{i=1}^k w_{n_i}.
\]
For a given integer $n\ge 1$ and $k\ge 1$ with $w_n^{(k)}>0$, the \emph{Gibbs distribution} $\mathbb{P}_n^{(k)}$ on ordered $k$-tuples $(n_1,\dots,n_k)$
of non-negative integers summing to $n$ is defined by
\begin{align*}
	\mathbb{P}_n^{(k)}(n_1,\dots,n_k) 
	= \frac{1}{w_n^{(k)} }  \prod_{i=1}^k w_{n_i}.
\end{align*}
Mixing over $k$ with the weights $(v_k)_{k\ge 0}$ yields the \emph{Gibbs partition model}:
\begin{align*}
	\mathbb{P}_n(n_1,\dots,n_k) 
	= \frac{v_k \prod_{i=1}^k w_{n_i}}{u_n}.
\end{align*}
We only consider values of $n$ for which $0<u_n<\infty$, ensuring that the Gibbs partition model \(\mathbb{P}_n\) is a well-defined probability measure. We let $(K_1, \ldots, K_{N_n})$ denote a sequence of  random non-negative integers that follows the  $\mathbb{P}_n$-distribution. The dependence of the $K_i$ on $n$ is implicit. We refer to the coordinates of this sequences as components. Hence  $N_n$ denotes the number of components. We let $K_{(1)} \ge K_{(2)} \ge \ldots$ denote the components ordered according to their size.


\vspace{1em}
\noindent\textbf{Standing assumptions.}
In this work, we focus on the case where the weight sequences satisfy the following conditions.
\begin{enumerate}
	\item \emph{Condition on component weights:} there exists a slowly varying function $L_w$ and a constant $\rho_w>0$ such that
	\begin{align}
		\label{eq:condw}
		w_n \sim L_w(n) n^{-1} \rho_w^{-n}, \qquad W(\rho_w) \in ]0,\infty[.
	\end{align}
	\item \emph{Condition on component-count weights:} for some $\alpha>-1$ and a slowly varying function $L_v$,
	\begin{align}
		\label{eq:condv}
		v_n \sim L_v(n) n^{\alpha-1}  W(\rho_w)^{-n}.
	\end{align}
\end{enumerate}

Throughout the following we assume that~\eqref{eq:condw} and~\eqref{eq:condv} hold. Let $X \ge 0$ denote a random non-negative integer with probability generating function
\begin{equation}
\label{eq:defX}
\Exb{z^X} = W(\rho_w z) / W(\rho_w).
\end{equation}
By Karamata's theorem \cite[Rem. 1.2.7]{Mikosch1999_EURANDOM_99-013}, $\Prb{X>x}$ is slowly varying, and satisfies
\begin{align}
	\label{eq:lox}
    L_w(n) = o(\Prb{X>n}). 
\end{align}

\subsection{Main results}
Our first main theorem describes the asymptotic distribution of the number of components.

\begin{theorem}
	\label{te:main1}
	Uniformly for all integers $k = x / \Prb{X>n}$ with $x$ restricted to a compact subset of $]0, \infty[$ we have
	\[
		\Prb{N_n = k}  \sim \Prb{X>n}x^\alpha \exp(-x) / \Gamma(\alpha +1).
	\]
	Furthermore,
	\[
		N_n  \Prb{X>n} \convd Z
	\]
	for a $\mathrm{Gamma}(\alpha+1,1)$-distributed random variable $Z$. For each $q \ge 1$
	\[
		\Exb{N_n^q} \sim \Prb{X>n}^q \frac{\Gamma(q+\alpha+1)}{\Gamma(\alpha+1)}.
	\]
\end{theorem}

In our proof we make use of a framework for random walk  developed in~\cite{zbMATH05836300,zbMATH06031483,zbMATH06666232}. 
With the number of components under control, we turn to the extremal component sizes $K_{(1)} \ge K_{(2)} \ge \ldots$. We discover a condensation phenomenon, producing a unique giant component of size $n$ up to second order asymptotics of smaller order. The non-maximal components behave like independent copies of $(X \mid X \le n)$. In order to state this formally, we define the 'delete first maximum' operator
\[
\tau: \bigcup_{k \ge1} \ndN_0^k \to \bigcup_{k \ge 0} \ndN_0^k, \quad (x_1,\dots,x_k\bigr) \mapsto (x_1,\dots,x_{J-1},\,x_{J+1},\dots x_k),
\]
where $J=\min\{\,j: x_j=\max_{1\le i\le k}x_i\}$.  Let $\tilde{X}_1, \tilde{X}_2, \ldots$ denote independent copies of $(X \mid X \le n)$ that are independent of $N_n$.
\begin{theorem}
	\label{te:main2}
	We have
		\[
	d_{\mathrm{TV}}\Big(
	\mathcal{L}\big(\tau(K_1,\dots,K_{N_n})\big),
	\mathcal{L}\big(\tilde{X}_1,\dots, \tilde{X}_{N_n-1}\big)
	\Bigr)
	\;\longrightarrow\;0.
	\]
	Consequently,
	\[
		K_{(1)} / n \convp 1, \qquad K_{(2)} = o_p(n).
	\]
\end{theorem}

A consequence of Theorem~\ref{te:main2} is that for all $\epsilon>0$ and all subsets $A_n$ with $\max(A_n) < (1-\epsilon)n $ the number
\[
F_A = \sum_{i=1}^{N_n} \one_{K_{i} \in A_n}
\]
of components with size in $A_n$ satisfies
\begin{align}
d_{\mathrm{TV}}\Big(
\mathcal{L}\big(F_{A_n}), \bin(N_n, \Prb{X \in A_n \mid X \le n})\Bigr)
\;\longrightarrow\;0.
\end{align}
The expectation of the binomial is asymptotically equivalent to $\frac{\Pr{X \in A_n}}{\Pr{X>n}}$. For example, if $A_n = A$ is constant,
\begin{align}
		F_{A} \Pr{X>n} \convdis \Pr{X \in A} Z.
\end{align}

In order to determine the higher order asymptotics of the largest component $K_{(1)}$ and at the same time the precise limiting behavior of the $i$th largest component $K_{(i)}$ for $i \ge 2$, we require an additional regularity assumption:
\begin{align}
	\label{eq:regc}
	\lim_{n \to  \infty}\frac{\Prb{X>n^t}}{\Prb{X>n}} \quad \text{exists for all $0<t<1$.}
\end{align}
This regularity condition may be characterized:
\begin{proposition}
	\label{pro:cong}
	\begin{enumerate}
		\item (Equivalence)
	Assumption~\eqref{eq:regc} holds if and only if there exists $\gamma \ge 0$ with
	\begin{align}
		\label{eq:regc1}
	\lim_{n \to \infty} \frac{\Prb{X>n^t}}{\Prb{X>n}} = t^{-\gamma}
	\end{align} for all $0<t<1$.
		\item (Sufficient condition) If there is $\gamma>0$ and some slowly varying function $L$ with
		\begin{align}
			\label{eq:regc2}
			L_w(n) = \frac{L(\log n)}{\log^{1+\gamma}{n}},
		\end{align}
		  then~\eqref{eq:regc1} holds with the same $\gamma$. In this case
		  \[
		  	\Prb{X>n} \sim W(\rho_w)^{-1} \frac{L(\log n)}{\gamma  \log^{\gamma}{n}}.
		  \]
	\item (Partially necessary condition) If $L_w$ is eventually monotone and~\eqref{eq:regc1} holds with $\gamma>0$, then~\eqref{eq:regc2} holds for some slowly varying function $L$.
	\end{enumerate}
\end{proposition}

Condition~\eqref{eq:regc} allows us to define a locally finite Borel measure $\Lambda$ on $]0,1]$ such that
\[
	\Lambda([a,b]) = a^{-\gamma} - b^{-\gamma}
\]
for all $0<a<b\le1$. 

\begin{theorem}
	\label{te:main3}
	Suppose that~\eqref{eq:regc} holds.
	The point process \[
	\sum_{i=2}^ {N_n} \delta_\frac{\log(K_{(i)})}{\log(n)}\one_{K_{(i)}>1}
	\]
	converges towards a Cox process directed by the random measure $Z\Lambda$ for a $\mathrm{Gamma}(\alpha+1,1)$ distributed random variable $Z$.
\end{theorem}


Suppose that~\eqref{eq:regc} holds with $\gamma>0$. Then the limiting point process in Theorem~\ref{te:main3} is simple, and we let $Y_{(1)}>Y_{(2)}> \ldots$ denote its ranked points. Since $N_n$ grows at a slowly varying scale by Theorem~\ref{te:main1}, it follows that
\[
	0 \le n - \sum_{i=1}^{m} K_{(i)} \le (N_n -m) K_{(m+1)} = o_p(K_{(m)}),
\]
yielding
\[
	K_{(1)} = n - \sum_{i=2}^{m} K_{(i)} + o_p(K_{(m)}).
\]
Using Skorokhod's representation theorem, we may find a probability space where the joint limits of the extremal component sizes and asymptotic bounds for the number of components hold almost surely. The remainder $n - \sum_{i=1}^{m} K_{(i)} \le (N_n -m) K_{(m+1)}$ is bounded from below by $K_{(m+1)} = n^{Y_{(m+1)} + o(1)}$ and from above by $(N_n -m) K_{(m+1)} = n^{o(1)} K_{(m+1)}  = n^{Y_{(m+1)} + o(1)}$. Hence, on this space we have for each $m \ge 2$
\[
	K_{(1)} = n - \sum_{i=2}^{m} n^{Y_{(i)} + o(1)}, \qquad K_{(m)} = n^{Y_{(m)} + o(1)}.
\]

\subsection*{Plan of the paper} In Section~\ref{sec:proofs} we collect the proofs of our main results. Specifically, we prove Theorem~\ref{te:main1} in Subsection~\ref{sec:numco}.   Theorem~\ref{te:main2} is shown in Subsection~\ref{sec:condensation}. Proposition~\ref{pro:cong} is proved in Subsection~\ref{sec:character}. We verify Theorem~\ref{te:main3} in Subsection~\ref{sec:approx}. 

In Section~\ref{sec:Apps} we provide applications of our main results to combinatorial structures and random walks. Our applications encompass random walks in the cube in Subsection~\ref{sec:randomwalkinthecube}, Delannoy walks in Subsection~\ref{sec:delannoywalks}, 
Dyck bridges and Hadamard products in Subsection~\ref{ssec:DyckBridges}, colored walks in Subsection~\ref{sec:coloredwalks}, Kreweras, lazy Kreweras and diagonal square lattice walks in Subsection~\ref{SubSec:squareWalks}, sampling without replacement urns in Subsection~\ref{sec:urn}, and card guessing games in Subsection~\ref{sec:cards}. We close the section with an outlook on further families of lattice paths in Subsection~\ref{sec:outlook}.

\subsection*{Notation}

We use the notation $\ndN_0 = \{0, 1, \ldots\}$ and $\ndN= \{1,2, \ldots\}$. For $n \in \ndN_0$ we set $[n] = \{1, \ldots, n\}$.
All unspecified limits are taken as $n \to \infty$. We let $\convd$ denote convergence in distribution, and $\convp$ convergence in probability. $d_{\mathrm{TV}}$ denotes the total variational distance. The law of a random variable $X$ is denoted by $\mathcal{L}(X)$. For a random real number $Y$, we let $o_p(Y)$ denote the product of $Y$ and an unspecified random variable that tends in probability to zero.  The $n$th coefficient of a power series $f(z)$ is denoted by $[z^n]f(z)$. For multivariate series $f(x_1, x_2, \ldots)$ we use analogous notation~$[x_1^{n_1} x_2^{n_2} \cdots]f(x_1, x_2, \ldots)$. Concerning our combinatorial applications, we give for concrete objects links to the On-Line Encyclopedia of Integer Sequences - \href{https://oeis.org/}{OEIS}, whenever possible.

\section{Proofs of our main results}
\label{sec:proofs}

\subsection{The number of components.}
\label{sec:numco}

We let $X_1, X_2, \ldots$ denote independent copies of $X$ and set $S_n = \sum_{i=1}^n X_i$. The following was shown in~\cite[Thm. 1.1]{zbMATH06666232}:

\begin{proposition}
	We have uniformly for all $k$ with $k L_w(n) \to 0$,
	\begin{align}
		\label{eq:qb1}
		\Prb{S_k = n} \sim k \Pr{X=n}(1 - \Prb{X>n})^k.
	\end{align}
	There exists  $c>0$ independent of $n$ such that for all large enough $n$ and $k \ge 1$
	\begin{align}
		\label{eq:qb2}
		\Prb{S_k = n} \le c k \Pr{X=n}(1 - \Prb{X>n})^k.
	\end{align}
\end{proposition}
To be precise,~\cite[Thm. 1.1]{zbMATH06666232} considers step-distributions that are almost surely positive and hence states~\eqref{eq:qb2} for $k \le n$. However, we may write
\begin{align*}
	\Prb{S_k = n} &= \sum_{j=1}^{\min(k,n)} \binom{k}{j} \Pr{X=0}^{k-j} \Pr{X>0}^j \Prb{\sum_{i=1}^j X_i = n \mid X_1, \ldots, X_j >0}.
\end{align*}
Applying~\cite[Thm. 1.1]{zbMATH06666232} to the sum of independent copies of $(X \mid X >0)$ and a quick calculation then yields~\eqref{eq:qb2}.

See also~\cite{zbMATH05836300} for  tail-probabilities of these sums, and~\cite{zbMATH06031483} for a strong renewal theorem.
\begin{proposition}
	\label{pro:partition}
	We have
	\[
	u_n \sim \rho_w^{-n} \Pr{X=n} L_v(1 / p_n) p_n^{-\alpha-1} \Gamma(\alpha+1).
	\]
\end{proposition}
\begin{proof}
	By definition of $X$, we have
	\begin{align}
		\label{eq:uequation}
		u_n \rho_w^n = \sum_{k \ge 0} v_k W(\rho_w)^k  \Pr{S_k = n}.
	\end{align}
	We write $p_n := \Prb{X>n}$ and set $k = x/p_n$ for integers $k$. Let $0<a<b$. By~\eqref{eq:lox} and~\eqref{eq:qb1} we have
	\begin{align}
		\label{eq:a1}
		\sum_{\frac{a}{p_n} < k < \frac{b}{p_n}} v_k W(\rho_w)^k  \Pr{S_k = n} &\sim \Pr{X=n} \sum_{\frac{a}{p_n} < k < \frac{b}{p_n}} L_v(k) k^{\alpha}  (1 - p_n)^k \\
		&\sim \Pr{X=n} p_n^{-\alpha} \sum_{\frac{a}{p_n} < k < \frac{b}{p_n}} L_v(x / p_n) x^\alpha \exp(-x) \nonumber \\
		&\sim \Pr{X=n} L_v(1 / p_n) p_n^{-\alpha-1} \int_{a}^b  x^\alpha \exp(-x)\,\mathrm{d}x. \nonumber
	\end{align}
	Fix $\epsilon>0$ sufficiently small so that $\alpha-\epsilon >1$. By~\eqref{eq:qb2} and the Potter bounds \cite{Mikosch1999_EURANDOM_99-013} there exist $C>0$ with 
	\begin{align}
		\label{eq:a2}
		&\sum_{k \notin [\frac{a}{p_n}, \frac{b}{p_n}]} v_k W(\rho_w)^k  \Pr{S_k = n} \\
		&\le c  \Pr{X=n} \sum_{k \notin [\frac{a}{p_n}, \frac{b}{p_n}]} L_v(k) k^{\alpha}  (1 - p_n)^k \nonumber \\
		&\le C \Pr{X=n} p_n^{-\alpha } \sum_{k \notin [\frac{a}{p_n}, \frac{b}{p_n}]}  x^{\alpha} \max(x^\epsilon, x^{-\epsilon}) \exp(-x) \nonumber \\
		&\sim  C \Pr{X=n} L_v(1 / p_n) p_n^{-\alpha-1} \int_{\ndR_{> 0} \setminus ]a,b[}  x^\alpha \max(x^\epsilon, x^{-\epsilon}) \exp(-x)\,\mathrm{d}x. \nonumber
	\end{align}
	The integral converges. As $0<a<b$ were arbitrary, it follows by~\eqref{eq:uequation} that
	\begin{align*}
		u_n \rho_w^n \sim \Pr{X=n} L_v(1 / p_n) p_n^{-\alpha-1} \int_{0}^\infty  x^\alpha \exp(-x)\,\mathrm{d}x.
	\end{align*}
	This completes the proof.
\end{proof}

We are now ready to prove Theorem~\ref{te:main1}.
\begin{proof}
	We have
	\[
	\Prb{N_n=k} = \frac{v_k W(\rho_w)^k  \Pr{S_k = n}}{u_n \rho_w^n}.
	\]
	Writing $k=x / \Pr{X>n}$ with $x$ restricted to a compact subset of $]0,\infty[$, it follows by Proposition~\ref{pro:partition} and identical arguments as in~\eqref{eq:a1}
	\[
	\Prb{N_n=k} \sim \Prb{X>n}x^\alpha \exp(-x) / \Gamma(\alpha +1).
	\]
	This local limit theorem implies
	\[
	N_n \Prb{X>n} \convdis \mathrm{Gamma}(\alpha+1,1).
	\]
	In order to verify convergence of all moments note that for all $q>1$  we have by Proposition~\ref{pro:partition} and identical arguments as in~\eqref{eq:a2} that for some constants $C>0$ and $\epsilon>0$ with $\alpha-\epsilon>-1$
	\begin{align*}
		\sup_{n \ge 1} \Exb{ \left(N_n \Prb{X>n}\right)^q} &\le   \int_0^\infty  x^{\alpha+q} \max(x^\epsilon, x^{-\epsilon}) \exp(-x)\,\mathrm{d}x  < \infty.
	\end{align*}
	Consequently, all moments of $N_n \Prb{X>n}$ converge. This completes the proof.
\end{proof}

\subsection{Condensation in the balls in boxes model}
\label{sec:condensation}

We show that under conditioning $\{S_k=n\}$, the $k-1$ ``small'' summands remaining after deleting the unique maximum become asymptotically independent copies of $(X \mid X \le n)$. To this end we require that $k= k(n)$ doesn't grow too fast. Recall that $\tau$ denotes the 'delete first max' operator.

\begin{theorem}
	\label{te:cond}
	For each $n\ge1$, let $k=k(n)$ be an integer such that
	$
	k L_w(n) \to 0.
	$
	Denote by
	\[
	\mathcal{L}(\tau(X_1,\dots,X_k)\,\bigm\vert\,S_k=n\bigr)
	\quad\text{and}\quad
	\mathcal{L}((X_1,\dots,X_{k-1})\,\bigm\vert\,X_1\le n,\dots,X_{k-1}\le n\bigr)
	\]
	the corresponding conditional laws on $\mathbb{N}_0^{\,k-1}$.  Then, as $n\to\infty$,
	\[
	d_{\mathrm{TV}}\Big(
	\mathcal{L}\big(\tau(X_1,\dots,X_k)\,\bigm\vert\,S_k=n\bigr)\;,\;
	\mathcal{L}\big((X_1,\dots,X_{k-1})\,\bigm\vert\,X_i\le n\ \forall\,i\bigr)
	\Bigr)
	\;\longrightarrow\;0.
	\]
	Consequently, 
	\[
	\left( \max(X_1, \ldots, X_k) \mid S_k = n \right) /n \convp 1.
	\]
	
\end{theorem}

\begin{proof}
	By Karamata's theorem,
	\begin{align*}
		\Exb{ \sum_{i=1}^{k-1} X_i \,\,\Big\vert\,\, X_1, \ldots, X_{k-1} \le n }	&\sim k \Exb{ X \mid X \le n} \\
		&= \frac{k}{\Prb{X \le n}} \sum_{i=1}^n \Prb{X=i} i \\
		&\sim k n L_w(n) / W(\rho_w)  \\
		&=o(n).
	\end{align*}
	By Markov's inequality and~\eqref{eq:lox}, it follows that there exists a sequence of integers $r_n = o(n)$ such that
	\begin{align}
		\label{eq:on}
		\Prb{ \sum_{i=1}^{k-1} X_i \ge r_n \,\,\Big\vert\,\, X_1,\ldots, X_{k-1} \le n }	\;\longrightarrow\;0.
	\end{align}
	Let
	\[
	\bm{s}=(s_1,\dots,s_{k-1})\in\mathbb{N}^{\,k-1}
	\quad\text{with}\quad
	\sum_{i=1}^{\,k-1}s_i \;=:\;s \;<\;r_n.
	\]
	Define
	\[
	p_\tau(\bm{s})\;:=\;\Prb{\tau(X_1,\dots,X_k)=s\;\mid\;S_k=n}\]
	and \[
	p_{\le n}(\bm{s})\;:=\;\Prb{(X_1,\dots,X_{k-1})=s\;\mid\;X_1, \ldots, X_{k-1} \le n}.
	\]
	By~\eqref{eq:on} it suffices to prove
	\begin{align}
		\label{eq:toverify}
		\sup_{\bm{s}\,:\,s <r_n }\Bigl|1 - \frac{p_\tau(\bm{s})}{p_{\le n}(\bm{s})} \Bigr|\;\xrightarrow[n\to\infty]{}0.
	\end{align}

	Note that, if $\tau(X_1,\dots,X_k)=s$ and $S_k=n$, then the deleted coordinate must equal $n-s$, and that the maximal value can occur at any of the $k$ positions.  Hence, for all sufficiently large $n$ and each  $\bm{s}$ with $\sum s_i=s<r_n$,
	\[
	\Prb{\tau(X_1,\dots,X_k)=\bm{s},\;S_k=n}
	\;=\;
	k\;\Bigl[\prod_{i=1}^{k-1}\Prb{X=s_i}\Bigr]\;\Prb{X=n-s}.
	\]
	Since \(\Prb{S_k=n}\) is the normalising constant under the conditioning $\{S_k=n\}$,
	\[
	p_\tau(\bm{s})
	=\frac{\Prb{\tau(X_1,\dots,X_k)=\bm{s},\;S_k=n}}{\Prb{S_k=n}}
	=\frac{k\,\bigl[\prod_{i=1}^{k-1}\Prb{X=s_i}\bigr]\;\Prb{X=n-s}}{\Prb{S_k=n}}.
	\]
	On the other hand, 
	\[
	p_{\le n}(\bm{s})
	=\frac{\prod_{i=1}^{k-1}\Prb{X=s_i}}{\bigl[\Prb{X\le n}\bigr]^{\,k-1}}.
	\]
	Hence, using~\eqref{eq:qb1} and $s \le r_n = o(n)$
	\begin{align*}
		\frac{p_\tau(\bm{s})}{\,p_{\le n}(\bm{s})\,}
		&=\;
		k\;\frac{\Prb{X=n-s}}{\Prb{S_k=n}}\;\bigl[\Prb{X\le n}\bigr]^{\,k-1} \\
		&\sim \frac{\Prb{X=n-s}}{\Prb{X=n}} \\
		&\sim 1.
	\end{align*}
	This verifies~\eqref{eq:toverify} and hence completes the proof.
\end{proof}

Compare with~\cite{MR2775110} and~\cite{kortchemski2025condensationsubcriticalcauchybienayme} for big-jump principles where the non-maximal components asymptotically no longer depend on $n$.

We now prove Theorem~\ref{te:main2} by viewing the Gibbs partition model as a mixture of the balls in boxes model with the number of Boxes controlled via Theorem~\ref{te:main1}.

\begin{proof}[Proof of Theorem~\ref{te:main2}]
	By~\eqref{eq:lox} there exists a sequence $t_n$ with $\Pr{X>n}^{-1} << t_n << L_w(n)^{-1}$.  By Theorem~\ref{te:cond}, it follows that
	\begin{align}
		\sup_{k \le t_n} d_{\mathrm{TV}}\Big(
		\mathcal{L}\big(\tau(X_1,\dots,X_k)\,\bigm\vert\,S_k=n\bigr)\;,\;
		\mathcal{L}\big((X_1,\dots,X_{k-1})\,\bigm\vert\,X_i\le n\ \forall\,i\bigr)
		\Bigr)
		\;\longrightarrow\;0.
	\end{align}
	Indeed, otherwise there would exist an $\epsilon>0$ and a sequence $k=k(n)$ with $k(n) \le t_n$ such that the total variational distance between the two laws is larger than $\epsilon$ for infinitely many $n$, a clear violation of the statement of Theorem~\ref{te:cond}.
	
	For any integer $k$ with $\Prb{N_n=k}>0$ we have
	\[
	( (K_1, \ldots, K_{N_n}) \mid N_n=k) \eqdist (X_1, \ldots, X_k \mid S_k = n).
	\]
	By Theorem~\ref{te:main1} we have that $N_n < t_n$ with probability tending to $1$ as $n \to \infty$. Hence
	\[
	d_{\mathrm{TV}}\Big(
	\mathcal{L}\big(\tau(K_1,\dots,K_{N_n})\big),
	\mathcal{L}\big(\tilde{X}_1,\dots, \tilde{X}_{N_n-1}\big)
	\Bigr)
	\;\longrightarrow\;0.
	\]
	The sequence $r_n$ in~\eqref{eq:on} may be chosen to be independent of $k$, hence we obtain
	\[
	K_{(1)} / n \convp 1.
	\]	
\end{proof}

\subsection{Characterising further regularity assumptions}
\label{sec:character}

In this section we characterise the regularity assumption~\eqref{eq:regc} by proving Proposition~\ref{pro:cong}.

\begin{proof}[Proof of Proposition~\ref{pro:cong}]
	If~\eqref{eq:regc} holds, then the function $g: ]0,1] \to \ndR_{\ge 0}, t \mapsto \lim_{n \to  \infty}\frac{\Prb{X>n^t}}{\Prb{X>n}}$ decreases monotonically, $g(1)=1$, and $g(st) = g(s)g(t)$ for all $0<t,s \le 1$. We may extend $g$ to a multiplicative function  $g: ]0,\infty[ \to ]0,\infty[$ by setting $g(t) = g(1/t)$ for $t>1$. This way,~\cite[Thm. 1.1.9]{zbMATH00043570} applies, yielding $g(t) \equiv t^{-\gamma}$ for some $\gamma \in \ndR$. Since $g$ decreases monotonically, we have $\gamma \ge 0$.  
	
	Is is easy to see that~\eqref{eq:regc} implies the same limit when $n$ grows along all real numbers (as opposed to the integers). This implies that $\Pr{X> \exp(y)}$ varies regularly at infinity with index $-\gamma$. If $L_w$ is eventually monotone, then \begin{align*}
		\Prb{X>\exp(y)} &\sim \sum_{k > \exp(y)} L_w(k) k^{-1} / W(\rho_w) \\
		&\sim \int_{\exp(y)}^\infty L_w(x) x^{-1} / W(\rho_w)\,\mathrm{d}x \\
		&= \int_{y}^\infty L_w(\exp(z))  / W(\rho_w)\,\mathrm{d}z
	\end{align*}
	as $y \to \infty$.
	By the monotone density theorem,
	\begin{align*}
		L_w(\exp(y)) &\sim W(\rho_w) \Prb{X>\exp(y)} \frac{\gamma}{y}
	\end{align*}
	as $y \to \infty$. Hence $L_w(\exp(y))$ varies regularly at infinity with index $-\gamma -1$. In other words, there exists a slowly varying function $L$ with
	\[
	L_w(y) \sim \frac{L(\log y)}{\log^{1+\gamma}(y)}
	\]
	as $y \to \infty$.
	
	Conversely, if $L_w(y) \sim \frac{L(\log y)}{\log^{1+\gamma}(y)}$ as $y \to \infty$ for some $\gamma>0$ and some slowly varying function $L$, then by an elementary calculation and Karamata's theorem
	\begin{align*}
		\Prb{X>n} &\sim W(\rho_w)^{-1} \int_{n}^\infty \frac{L(\log y)}{y \log^{1+\gamma}(y)}\,\mathrm{d}y \\
		&= W(\rho_w)^{-1} \int_{\log n}^\infty \frac{L(z)}{ z^{1+\gamma}}\,\mathrm{d}y \\
		&\sim  W(\rho_w)^{-1} \frac{L(\log n)}{\gamma  \log^{\gamma}{n}}.
	\end{align*}
	This implies
	\[
	\lim_{n \to \infty} \frac{\Prb{X>n^t}}{\Prb{X>n}} = t^{-\gamma}.
	\]
	Hence the proof is complete.
\end{proof}

\subsection{Non-maximal components in the balls in boxes model}
\label{sec:approx}

In the previous section we showed that the balls in boxes model exhibits a unique giant component with size $n$ up to higher order asymptotics that scale at a slower rate. The next theorem determines these higher order asymptotics and the rate of all maximal components.  We let $\delta_{(\cdot)}$ denote the Dirac measure.

\begin{theorem}
	Let $k_n$ denote a sequence  of positive integers with $k_n \Prb{X>n} \to x>0$. With \[
	(\widehat{X}_1, \ldots, \widehat{X}_{k_n-1}) := \left(\tau(X_1, \ldots, X_{k_n}) \mid S_{k_n} = n\right),
	\]
	the  point process \[
	\sum_{i = 1}^{k_n -1}\delta_\frac{\log(\widehat{X}_i)}{\log(n)}\one_{\widehat{X}_i>1}
	\]
	converges towards an inhomogeneous Poisson point process on $]0,1]$ with intensity measure $x\Lambda$.
\end{theorem}
\begin{proof}
	Let $\tilde{X}_i$, $i \ge 1$ denote independent copies of $(X \mid X \le n)$. By Theorem~\ref{te:cond} it suffices to show convergence of the point process
	\begin{align}
		\label{eq:modproc}
		\xi_n := \sum_{i = 1}^{k_n -1}\delta_\frac{\log(\tilde{X}_i)}{\log(n)}\one_{\tilde{X}_i>1}.
	\end{align}
	Let $\psi: ]0,1] \to \ndR_{\ge 0}$ be a  continuous function with compact support.  We compute the Laplace functional
	\begin{align}
		\label{eq:lapl}
		\Exb{\exp\left(- \int \psi(t) \,\mathrm{d}\xi_n \right) } &= \Exb{\exp\left(- \sum_{i=1}^{k_n-1} \psi\left( \frac{\log(\tilde{X}_i)}{\log n}\right)\one_{\tilde{X}_i>1} \right)} \\
		&= \Exb{ \exp\left(- \psi\left(\frac{\log(\tilde{X}_1)}{\log n}\right)\one_{\tilde{X}_1>1} \right) }^{k_n -1}. \nonumber
	\end{align}
	Using $	\frac{\Prb{X>n^t}}{\Prb{X>n}}\sim g(t)$ for all $t \in ]0,1]$, it follows for $0<a<b\le1$ that
	\begin{align*}
		\Prb{a \le \frac{\log(\tilde{X}_1)}{\log n} \le b } &\sim \Prb{n^a \le \tilde{X}_1 \le n^b }  \\
		&\sim \Prb{\tilde{X}_1 \ge n^a} - \Prb{\tilde{X}_1 > n^b} \\
		&\sim \Prb{X>n} \Lambda([a,b]).
	\end{align*}
	Let $\nu_n$ denote the restriction of $\Prb{X>n}^{-1} \mathcal{L}\left( \frac{\log(\tilde{X}_1)}{\log n} \one_{\tilde{X}_1>0} \right)$ (a measure on $[0,1]$) to a measure on $]0,1]$. Then $\nu_n$ converges in the vague topology towards $\Lambda$ and
	\begin{align*}
		&\Prb{X>n}^{-1}  \Exb{ 1 - \exp\left(-  \psi\left(\frac{\log(\tilde{X}_1)}{\log n} \right)\one_{\tilde{X}_1>1}\right) } \\
		&= \int_{]0,1]}(1 - \exp(-\psi(t))) \,\nu_n(\mathrm{d}t) \\
		&\to  \int_{]0,1]}(1 - \exp(-\psi(t))) \,\Lambda(\mathrm{d}t).
	\end{align*}
	Using $k_n \sim x / \Prb{X>n}$ it follows by~\eqref{eq:lapl} that
	\begin{align}
		\label{eq:conv}
		\Exb{\exp\left(- \int \psi(t) \,\mathrm{d}\xi_n \right) } \to \exp\left( -x \int_{]0,1]}(1 - \exp(-\psi(t))) \,\Lambda(\mathrm{d}t)  \right).
	\end{align}
	By~\cite[Lem. 3.1, Thm. 4.11]{zbMATH06684669} it follows that $\xi_n$ converges in distribution with respect to the vague topology towards a Poisson process with intensity measure $x\Lambda$. 
\end{proof}

We now prove Theorem~\ref{te:main3} again by modelling  Gibbs partitions as a mixture of the balls in boxes model with the box number  controlled via Theorem~\ref{te:main1}.

\begin{proof}[Proof of Theorem~\ref{te:main3}]
	By Theorem~\ref{te:main2} it suffices to show convergence of the point process
	\[
	\Xi_n := \sum_{i=1}^{N_n-1} \delta_\frac{\log(\bar{X}_i)}{\log(n)}\one_{\bar{X}_i>1}.
	\]
	Note that when we condition on $N_n =k$, then this is distributed like the point process $\xi_n$ defined in~\eqref{eq:modproc}.

	Let $\epsilon>0$ be given. By Theorem~\ref{te:main1} we may choose $0<a<b$ sufficiently large such that $N_n \Prb{X>n} \in [a,b]$ with probability larger than $1-\epsilon$ for all sufficiently large~$n$.
	
	For each $k$ in the interval $\Prb{X>n}^{-1}[a,b]$ we write $k= x_k / \Prb{X>n}$ so that $x_k \in [a,b]$. Let $\xi: ]0,1] \to \ndR_{\ge0}$ denote a bounded continuous function with compact support. By~\eqref{eq:conv} we have
	\[
	\Exb{\exp\left(- \int \psi(t) \,\mathrm{d}\Xi_n \right) \,\Big\vert\, N_n = k} \sim \exp\left( -x_k \int_{]0,1]}(1 - \exp(-\psi(t))) \,\Lambda(\mathrm{d}t)  \right)
	\]
	and it is easy to see from the proof of~\eqref{eq:conv} that this asymptotic equivalence is uniform in $k$ for $x_k$ restricted to $[a,b]$. It follows from Theorem~\ref{te:main1} that for some error term $R_n$ satisfying $|R_n|<\epsilon$ for large enough $n$ we have
	\begin{align*}
		&\Exb{\exp\left(- \int \psi(t) \,\mathrm{d}\Xi_n \right)} \\
		&= R_n + o(1) + \sum_{k : x_k \in[a,b]}\Prb{X>n} \frac{ x_k^\alpha \exp(-x_k)}{  \Gamma(\alpha +1)}  \exp\left( -x_k \int_{]0,1]}(1 - \exp(-\psi(t))) \,\Lambda(\mathrm{d}t)  \right) \\
		&=  R_n + o(1) + \Exb{  \exp\left( -Z \int_{]0,1]}(1 - \exp(-\psi(t))) \,\Lambda(\mathrm{d}t)  \right), a \le Z \le b } \\
		&=  2R_n + o(1) + \Exb{  \exp\left( -Z \int_{]0,1]}(1 - \exp(-\psi(t))) \,\Lambda(\mathrm{d}t)  \right)}.
	\end{align*}
	Since $\epsilon>0$ was arbitrary it follows that
	\begin{align*}
		\Exb{\exp\left(- \int \psi(t) \,\mathrm{d}\Xi_n \right)} \to \Exb{  \exp\left( -Z \int_{]0,1]}(1 - \exp(-\psi(t))) \,\Lambda(\mathrm{d}t)  \right)}.
	\end{align*}
	By~\cite[Lem. 3.1, Thm. 4.11]{zbMATH06684669} the proof is complete.
\end{proof}

\section{Applications\label{sec:Apps}}
In the following we present applications of our main results, Theorems~\ref{te:main1},~\ref{te:main2} and~\ref{te:main3}, to many different combinatorial models. 
For all models, we first describe the counting problem in terms of a composition scheme. Then, 
we use singularity analysis of generating functions~\cite{FlaOd1990,FlaSed} to derive
the asymptotics of the component weights~\eqref{eq:condw}, as well as the component-count weights~\eqref{eq:condv}. This allows to ensure the applicability of our main results by checking our standing assumptions and also the growth condition~\eqref{eq:regc2}.

\subsection{A random walk in the cube}
\label{sec:randomwalkinthecube}
We study a simple random walk in the cube with unit steps $\myvec{-1\\0\\0}, \myvec{0\\-1\\0},\myvec{0\\0\\-1}$. The starting point is the upper corner $(n,n,n)$ and the walk ends at the origin $(0,0,0)$. 

Our main interest is the number of times the random walk returns to the space diagonal $x=y=z$ of the cube. This random walk generalizes the classical Dyck paths\footnote{We note in passing that, strictly speaking, such lattice paths are usually known as Dyck bridges and refer the interested reader to the fundamental work of Banderier and Flajolet~\cite{BanderierFlajolet2002}.} from $(n,n)$ to $(0,0)$. There, the expected number of returns to the diagonal of the square $x=y$ is proportional to $\sqrt{n}$. In our case, dimension three, we will observe a growth order of $\log(n)$. Let $\mathcal{P}$ denote the paths from the corner $(n,n,n)$, $n\in\ndN_0$ to the origin with unit steps, staying inside the cube. We are interested in the number of returns to the diagonal for each such path. Let $a_{n,j}$, $j\ge 1$ denote the number of paths of length $3n$ from $(n,n,n)$ to the origin with exactly $j$ returns to the diagonal. Moreover, let $a_n$ denote the total number of such paths, 
\[
a_n=\sum_{j\ge 1}a_{n,j}=\binom{3n}{n,n,n},
\]
where the total number follows directly from a standard counting argument.
Let $N_n$ the random variable, counting the returns to the diagonal under the uniform distribution:
\[
\Prb{N_n=j}=\frac{a_{n,j}}{a_n},\quad j\ge 1.
\]
We remark that $(a_n)$ is the sequence \href{https://oeis.org/A006480}{OEIS A006480}. 
Let $\rt$ denote the corresponding parameter, $\rt\colon \mathcal{P}\to\ndN$.
The generating function 
\begin{equation}
	\label{def:bvGF}
	A(z,q)=\sum_{p\in\mathcal{P}}z^{|p|/3}q^{\rt(p)}=\sum_{n\ge 0}z^n\sum_{j\ge 1}a_{n,j}q^j=\sum_{n\ge 0}a_n\Exb{q^{N_n}}z^n,
\end{equation}
will be obtained using the so-called arch decomposition of paths. Here, $|p|$ denotes the length of the path $p\in\mathcal{P}$.
We can readily set a recurrence relation for these numbers by considering their first return to the diagonal. For $j>1$
\begin{equation}
	\label{eqn:recDiag1}
	a_{n,j}=\sum_{k=1}^{n-1}a_{k,j-1}a_{n-k,1},\quad m>0,
\end{equation}
with initial conditions
\begin{equation}
	\label{eqn:recDiagInitial}
	a_{0,j}=
	\begin{cases}
		1,\quad j=0,\\
		0,\quad j>0.
	\end{cases}
\end{equation}
For $j=1$ we use a different kind of recurrence relation. We observe that the total number of paths
$a_n=\binom{3n}{n,n,n}$ from $(n,n,n)$ to the origin can be decomposed according to the first return to the diagonal:
\begin{equation}
	\label{eqn:recDiag2}
	a_{n}=\sum_{k=1}^{n-1}a_{k}a_{n-k,1},\quad n>1.
\end{equation}
Let $P(z)$ denote the generating function of all paths, given by
\[
P(z)=A(z,1)=\sum_{n\ge 0}a_n z^n=\sum_{n\ge 0}\binom{3n}{n,n,n}z^n. 
\]
We point out that $P(z)$ can be represented in terms of an Euler-Gau\ss{} hypergeometric function
\begin{equation}
	\label{def:hypergeom}
	{}_2F_1(a,b;c;z) = \sum_{n=0}^\infty \frac{\auffak{a}n \auffak{b}n}{\auffak{c}n} \frac{z^n}{n!},
\end{equation}
where $\auffak{x}n=x(x+1)\dots (x+(n-1))$ denotes the $n$th rising factorial\footnote{A different often used notation is the so-called Pochhammer symbol $(x)_n=\auffak{x}{n}$}, 
as 
\begin{equation}
	\label{P:hypergeom}
	P(z)=\sum_{n\ge 0}\binom{3n}{n,n,n}z^n={}_2F_1\big(\frac13,\frac23;1;27z\big),
\end{equation}
with radius of convergence $\rho=\frac1{27}$, which follows directly by the ratio test. This representation will be of particular interest in our subsequent asymptotic and probabilistic considerations.
Let 
\[
A_j(z)=\sum_{n=j}^{\infty}a_{n,j}z^{n},\quad j\ge 1.
\]
The recurrence relation~\eqref{eqn:recDiag2} leads to
\begin{equation}
	\label{eqn:recDiag3}
	P(z)-1=A_1(z)P(z),\quad A_1(z)=1-\frac{1}{P(z)}.
\end{equation}
From~\eqref{eqn:recDiag1} we obtain
\[
A_j(z)=A_{j-1}(z)A_1(z),
\]
and by iteration
\[
A_j(z)=A_1(z)^j,\quad j\ge 1,
\]
also valid for $j=0$, by means of~\eqref{eqn:recDiagInitial}. Note that the formula above also leads to~\eqref{eqn:recDiag3},
as we have
\[
a_n=\sum_{j\ge 0}a_{n,j},\quad \sum_{j\ge 0}A_j(z)=\frac{1}{1-A_1(z)}=P(z).
\]
Finally, we combine the previous computations to obtain the desired bivariate generating function~\eqref{def:bvGF}.
\begin{lemma}
	\label{the:DiagToDiag}
	The bivariate generating function $A(z,q)=\sum_{n\ge 0}\sum_{j\ge 0}a_{n,j}q^j z^n$ 
	of the number of paths of length $3n$ from $(n,n,n)$ to the origin with exactly $j$ returns to the diagonal is given by
	\[
	A(z,q)=\frac{1}{1-q A_1(z)},
	\]
	with $A_1(z)= 1-\frac{1}{P(z)}$ and $P(z)={}_2F_1\big(\frac13,\frac23;1;27z\big)$. 
\end{lemma}
\begin{remark}
	We point out that this construction is actually true for all dimensions $m\ge 2$. 
	For $m\ge 4$ there are only a finite number of returns, as $n$ tends to infinity,
	which follows from classical results~\cite{FlaSed}.
\end{remark}
The bivariate generating function $A(z,q)$ can be interpreted in terms of symbolic combinatorics.
The set of all paths $\mathcal{P}$ in the cube from an upper corner to the origin can be obtained as a sequence $\Seq$ of arches $\mathcal{A}$, corresponding to the returns to the diagonal. In addition, the variable $q$ indicates the number of arches. This leads to the symbolic equation $\mathcal{P}=\Seq(q\mathcal{A})$. We collect an enumerative result below.
\begin{corollary}
	\label{coroll:noVisits}
	The sequence $(a_{n,1})_{n\ge 1}$ counts the number of paths of length $3n$ from $(n,n,n)$ to the origin, 
	never returning to the diagonal $x=y=z$ before the origin, starts with
	\[
	1,6,54,816,14814, 295812,6262488,137929392,3125822238,\dots
	\]
	and is the sequence \href{https://oeis.org/A378026}{OEIS A378026}.
	It satisfies the asymptotics
	\[
	a_{n,1}\sim \frac{2\pi}{\sqrt{3}}\cdot \frac{27^n}{n \log^2n},
	\]
	as $n\to\infty$.
\end{corollary}

In order to prove Corollary~\ref{coroll:noVisits} we use singularity analysis of generating functions~\cite{FlaSed,FlaOd1990} and determine the singular structure of $A_1(z)$, as well as $P(z)$. This allows to obtain the asymptotic expansions of $A_1(z)$. Moreover, these expansions allow to apply our main results, Theorems~\ref{te:main1},~\ref{te:main2} and~\ref{te:main3}
as condition~\eqref{eq:regc2} of Proposition~\ref{pro:cong} is readily verified with $\gamma=1$. 
\begin{lemma}
	\label{lem:SingularPA}
	The functions $P(z)=\sum_{n\ge 0}\binom{3n}{n,n,n}z^n$ and $A_1(z)$ both have the dominant singularity $\rho=\frac1{27}$. They satisfy for $z$ near $\rho$ the singular expansions
	\[
	P(z)\sim \frac{\sqrt{3}}{2\pi}\cdot \log\left(\frac{1}{1-\frac{z}\rho}\right),\quad A_1(z)\sim 1-\frac{1}{\frac{\sqrt{3}}{2\pi}\cdot \log\left(\frac{1}{1-\frac{z}\rho}\right)}.
	\]
\end{lemma}
\begin{proof}[Proof of Lemma~\ref{lem:SingularPA} and Corollary~\ref{coroll:noVisits}]
	The hypergeometric functions~\eqref{def:hypergeom} are amendable for singularity analysis, see~\cite{FlaSed}, as they are $\Delta$–continuable.
	We use the asymptotics of $\binom{3n}{n,n,n}$, obtained using Stirling's formula, 
	and Stirling's formula for the factorials,
	\begin{equation}
		\label{eqn:Stirling}
		n!\sim \frac{n^n}{e^n}\sqrt{2\pi n},
	\end{equation}
	applied to $\binom{3n}{n,n,n}$ to obtain the desired singular structure of $P(z)$, matching the asymptotics of the sequence to the singular structure.
	We note in passing that this procedure can also be used to obtain more precise (singular) expansions.
	This leads to 
	\begin{equation}
		\label{eqn:AsymptTotalNumber}
		\binom{3n}{n,n,n} \sim \frac{\sqrt{3}}{2\pi}\cdot \frac{27^n}{n}\quad \Leftrightarrow P(z)\sim \frac{\sqrt{3}}{2\pi}\cdot \log\left(\frac{1}{1-\frac{z}\rho}\right)
	\end{equation}
	for $z$ near $\rho=1/27$. The expansion of $A_1(z)$ follows directly from the stated result for $P(z)$. This also implies asymptotics for the sequence $(a_{n,1})$ by singularity analysis~\cite[page 388]{FlaSed}
	\[
	a_{n,1}=[z^n]A_1(z) \sim [z^n] -\frac{1}{\frac{\sqrt{3}}{2\pi}\cdot \log\left(\frac{1}{1-\frac{z}\rho}\right)}\sim \frac{1}{\frac{\sqrt{3}}{2\pi}}\cdot \frac{1}{\rho^n n \log^2n}.
	\]
\end{proof}

Using our enumerative results, we also identify the role of the random variable $X$ of~\eqref{eq:defX}; we point out that a similar interpretation may also be obtained in subsequent applications. 

\begin{proposition}
	Let $X_n$ denote the random variable that encodes the position of the first return to the diagonal. The probability mass function of $X_n$ is given by
	\[
	\Prb{X_n=k}=\frac{a_{n-k,1}\cdot \binom{3k}{k,k,k}} {\binom{3n}{n,n,n}},\quad 0\le k\le n-1.
	\]
	The random variable has a discrete limit law for $n\to\infty$:
	\[
	n-X_n \to X,\quad \Prb{X=\ell}= \frac{a_{\ell,1}}{3^{3\ell}},\quad \ell \ge 0.
	\]  
\end{proposition}

In the end, we also note that our result for $F_A$ easily implies a double phase transition - 
continuous to discrete to degenerate - for the components counts $N_{n,j}$ of a certain size $j$, in other words $A_n=j=j(n)$, such that $\sum_{j\ge 1}N_{n,j}=N_n$. We observe an intriguing critical growth range for this classical parameter of interest, given by $j\sim 2W(\frac{\sqrt{n}}{2})$,
where $W(z)$ denotes the Lambert-W-function. 
For results for $N_{n,j}$ in other settings we refer the reader to~\cite{BaKuWa,FlaSed,Stufler2022}. Here, we only note that $N_{n,j}/\lambda_{n,j}\to\text{NegBin}(2,p)$, for $\lambda_{n,j}=\frac{\log(n)}{j \log^{2}(j)}\to c>0$ with $p=1/(c+1)$. We leave the simple computations to the interested reader.

%

\subsection{Delannoy walks}
\label{sec:delannoywalks}
Delannoy paths in the cube start at $(n,n,n)$, $\n\in\ndN$ and end at the origin. As in our previous model, the allowed steps are unit steps in the negative direction, $\myvec{-1\\0\\0}, \myvec{0\\-1\\0},\myvec{0\\0\\-1}$ and additionally a step along the diagonal $\myvec{-1\\-1\\-1}$.
Again, we are interested in the number $N_n$ of returns to the space diagonal of the cube. To provide motivation for our model, we recall another famous model of lattice paths in the plane, namely Delannoy paths in the plane~\cite{Banderier_2005}. Such lattice paths start at the origin, end at $(n,n)$ with steps $(1,0)$, $(0,1)$ and $(1,1)$. Reversing the direction leads to paths from $(n,n)$ to the origin, with unit steps $-\vec{e}_1$, $-\vec{e}_2$ plus an additional step along the diagonal $-\vec{e}_1-\vec{e}_2=(-1,-1)$. The enumeration of Delannoy paths in the plane is a classical problem in combinatorics and has nice links to Schr\"oder paths and Schr\"oder numbers~\cite{GesselTalk}. In order to enumerated Delannoy paths in the cube, we first derive the generating function of the total number of paths. Let $d_{n,k}$ denote the number of paths from $(n,n,n)$, ending at the origin,
with $k$ steps $(-1,-1,-1)$ and let $d_n=\sum_{k=0}^n d_{n,k}$. For the first few values see \href{https://oeis.org/A081798}{OEIS A081798}. As we enter the $k$ steps into a path of length $n-k$ from $(n-k,n-k,n-k)$ to the origin.
We readily obtain
\[
d_{n,k}=\binom{3(n-k)}{n-k,n-k,n-k}\cdot\binom{3(n-k)+k}{k}.
\]
Recall $P(z)=\sum_{m\ge 0} \binom{3m}{m,m,m}z^m$, see~\eqref{P:hypergeom} and Lemma~\ref{lem:SingularPA} for its singular expansion. On the level of generating functions this leads to the equation
\[
D(z,q)=\sum_{n\ge 0}\sum_{k=0}^{n}z^n q^k d_{n,k} = \frac{1}{1-zq}P\big(\frac{z}{(1-zq)^3}\big)
=\sum_{m\ge 0} \binom{3m}{m,m,m}\frac{z^m}{(1-zq)^{3m+1}},
\]
and the desired generating function for the diagonal contacts is given by
\[
D(z,q,v)=\sum_{n\ge 0}\sum_{k=0}^{n}\sum_{j=0}^n z^n q^k v^j d_{n,k,j} 
=\frac{1}{1-v\big(1-\frac{1}{D(z,q)}\big)}.
\]
The counting series $D(z)=D(z,1)$ is given by
\[
D(z)=\frac{1}{1-z}P\big(\frac{z}{(1-z)^3}\big).
\]

The asymptotics of $D_n$ and the singular structure of $D(z)$ can be obtained from the singular structure of $P(z)$,
as $D(z)$ is itself a so-called extended functional composition~\cite{BaKuWa} of the form
$D(z)=M(z)\cdot P(H(z))$, where the outer function $P(z)=\sum_{m\ge 0} \binom{3m}{m,m,m}z^m$ determines the structure of the singular expansion. We analyze the equation
\[
\frac{z}{(1-z)^3}=\frac{1}{27},
\]
and obtain the dominant singularity as its real solution
\[
\rho= 1-3\cdot\sqrt[3]{\frac2{\sqrt{5}-1}}+3\cdot\sqrt[3]{\frac{\sqrt{5}-1}2}\approx 0.033444, 
\]
leading to the desired singular behavior, amendable to our main results
Theorems~\ref{te:main1},~\ref{te:main2} and~\ref{te:main3} with $\gamma=1$. 


At the end of this subsection we note a very similar model. The allowed steps are $\myvec{-1\\-1\\0}, \myvec{0\\-1\\-1},\myvec{-1\\0\\-1}$ and an additional step along the diagonal $\myvec{-1\\-1\\-1}$.
For the first few number of walks we refer to \href{https://oeis.org/A208425}{OEIS A208425}
The counting series of the total number of walks can again be expressed in terms of $P(z)={}_2F_1\big(\frac13,\frac23;1;27z\big)$:
\[
D(z)=\frac{1}{1-z}P\big(\frac{z^2}{(1-z)^3}\big).
\]
Again, our main results Theorems~\ref{te:main1},~\ref{te:main2} and~\ref{te:main3} apply with $\gamma=1$.

\subsection{Dyck bridges and Hadamard products}
\label{ssec:DyckBridges}
Recently, Li and Starr~\cite{Li2024} studied Dyck bridges, where the special case of a single bridge corresponds to the classical setting of a Dyck path in the square. We study $m$ different squares and $m$ walks of length $2n$, all with starting points $(n,n)$ and end points at the origin. We are interested in the number $N_n$ of simultaneous returns to the diagonals $x=y$ for all $m$ paths at the same time. Our main interest is in the case of $m=2$, a pair of Dyck bridges. Li and Starr observed a similar transition from $\sqrt{n}$ -- classical Dyck bridges-- to $\log(n)$ for two bridges, to a discrete number of returns for more than three paths. Similar to our study of the cube, we show how to obtain their result via a combination of our main results, as well as analytic combinatorics and Hadamard products, see Fill et al.~\cite{Fill2005}.

\smallskip

The generating function of all paths is given by a Hadamard product of the Dyck bridge generating function
\[
P(z)=\frac{1}{\sqrt{1-4z}}=\sum_{n\ge 0}\binom{2n}{n}z^n,
\]
such that
\[
P_m(z)= \sum_{n\ge 0}\binom{2n}{n}^m z^n=P^{\odot_m}(z)=P(z)\odot P(z)\odot \dots \odot P(z),
\]
with $P_1(z)=P(z)$. Here, the name-giving Hadamard product of generating functions
\[
A(z)=\sum_{n\ge 0}a_n z^n,
\quad 
B(z)=\sum_{n\ge 0}b_n z^n,
\qquad A(z)\odot B(z)=\sum_{n\ge 0}a_n b_n z^n.
\]
We remind the reader that singularity analysis can applied to Hadamard products~\cite[Theorem VI.10]{FlaSed}; in particular, the zigzag-algorithm~\cite[page 425]{FlaSed} allows to translate the asymptotics of $a_n\cdot b_n$ to the corresponding singular expansion.
By the arch decomposition, the generating function $P_m(z)$ consists of arches after simultaneous returns to the diagonals,
\[
\mathcal{P}_m=\Seq(v\mathcal{F}_m),\quad 
P_m(z,v)=\frac{1}{1-v F_m(z)},
\]
such that
\[
F_m(z)=1-\frac{1}{P_m(z)}.
\]

\begin{example}[A walk in the cube and elliptic integrals]
	\label{Ex:twoBridges}
	We associate a random walks in the cube to the sample paths
	for $m=2$, two Dyck bridges. It starts at $(n,n,n)$, ends at the origin and has steps
	\[
	\myvec{-1\\0\\0}, \myvec{0\\-1\\0},\myvec{-1\\0\\-1},
	\myvec{0\\-1\\-1}.
	\]
	The number $N_n$ of simultaneous returns of the Dyck bridges corresponds to the number of returns to the space diagonal. We postpone the simple bijection to Subsection~\ref{SubSec:squareWalks}. The total number of walks is given by the sequence $\binom{2n}{n}^2$, \href{https://oeis.org/A002894}{OEIS A002894}. It has a generating function, which can be written in terms of the complete elliptic integral of the first kind
	$K(k)$,
	\[
	K(k) = \int_0^{\frac{\pi}{2}} \frac{d\theta}{\sqrt{1-k^2 \sin^2\theta}} = \int_0^1 \frac{dt}{\sqrt{\left(1-t^2\right)\left(1-k^2 t^2\right)}},
	\]
	such that
	\[
	P_2(z)=\frac2\pi K(4\sqrt{z}), \quad F_2(z)=1-\frac{\pi}{2 K(4\sqrt{x})}.
	\]
	The sequence associated to $F_2(z)$ is \href{https://oeis.org/A054474}{OEIS A054474}.
	\begin{figure}[!htb]
		\centering
		\resizebox{5cm}{5cm}
		{
			\begin{tikzpicture}
				\begin{axis}[
					axis lines*=left,
					clip=false, no markers,
					xlabel=$x$,ylabel=$y$,  axis lines=middle,grid=major,
					xtick={0,...,5},  ytick={0,...,5},
					xmin=0, xmax=5,ymin=0, ymax=5
					] 
					\node [draw, fill, red,inner sep=2pt] at (axis cs: 5 , 5){}; 
					\node [draw, fill, red,inner sep=2pt] at (axis cs: 4 , 4){}; 
					\node [draw, fill, red,inner sep=2pt] at (axis cs: 3 , 3){}; 
					\node [draw, fill, red,inner sep=2pt] at (axis cs: 2 , 2){}; 
					\node [draw, fill, red,inner sep=2pt] at (axis cs: 1 , 1){}; 
					\node [draw, fill, red,inner sep=2pt] at (axis cs: 0 , 0){}; 
					\node [draw,circle,fill, blue,inner sep=1.2pt] at (axis cs: 4 , 5){};  
					\node [draw,circle,fill, blue,inner sep=1.2pt] at (axis cs: 3 , 4){};  
					\node [draw,circle,fill, blue,inner sep=1.2pt] at (axis cs: 3 , 3){};  
					\node [draw,circle,fill, blue,inner sep=1.2pt] at (axis cs: 3 , 2){};  
					\node [draw,circle,fill, blue,inner sep=1.2pt] at (axis cs: 2 , 2){};  
					\node [draw,circle,fill, blue,inner sep=1.2pt] at (axis cs: 2 , 1){};
					\node [draw,circle,fill, blue,inner sep=1.2pt] at (axis cs: 2 , 0){};  
					\node [draw,circle,fill, blue,inner sep=1.2pt] at (axis cs: 1 , 0){};  
					\node [draw,circle,fill, blue,inner sep=1.2pt] at (axis cs: 0 , 0){};  
					\addplot[red] coordinates {(0,0) (5,5)};
					\addplot[->, thick,  blue] coordinates {(5,5) (4,5)};
					\addplot[->, thick,  blue] coordinates {(4,5) (4,4)};
					\addplot[->, thick,  blue] coordinates {(4,4) (3,4)};
					\addplot[->, thick,  blue] coordinates {(3,4) (3,3)};
					\addplot[->, thick,  blue] coordinates {(3,3) (3,2)};
					\addplot[->, thick,  blue] coordinates {(3,2) (2,2)};
					\addplot[->, thick,  blue] coordinates {(2,2) (2,1)};
					\addplot[->, thick,  blue] coordinates {(2,1) (2,0)};
					\addplot[->, thick,  blue] coordinates {(2,0) (1,0)};
					\addplot[->, thick,  blue] coordinates {(1,0) (0,0)};
					\node [draw,circle,fill, blue,inner sep=2pt] at (axis cs: 0 , 0){}; 
					\node [draw,circle,fill, blue,inner sep=2.5pt] at (axis cs: 5 , 5){};  
					\node [draw,circle,fill, blue,inner sep=2.5pt] at (axis cs: 4 , 4){};  
				\end{axis}
			\end{tikzpicture} 
		}
		\quad
		\resizebox{5cm}{5cm}
		{
			\begin{tikzpicture}
				\begin{axis}[
					axis lines*=left,
					clip=false, no markers,
					xlabel=$x$,ylabel=$y$,  axis lines=middle,grid=major,
					xtick={0,...,5},  ytick={0,...,5},
					xmin=0, xmax=5,ymin=0, ymax=5
					] 
					\node [draw, fill, red,inner sep=2pt] at (axis cs: 5 , 5){}; 
					\node [draw, fill, red,inner sep=2pt] at (axis cs: 4 , 4){}; 
					\node [draw, fill, red,inner sep=2pt] at (axis cs: 3 , 3){}; 
					\node [draw, fill, red,inner sep=2pt] at (axis cs: 2 , 2){}; 
					\node [draw, fill, red,inner sep=2pt] at (axis cs: 1 , 1){}; 
					\node [draw, fill, red,inner sep=2pt] at (axis cs: 0 , 0){}; 
					\node [draw,circle,fill, blue,inner sep=1.2pt] at (axis cs: 5 , 4){};  
					\node [draw,circle,fill, blue,inner sep=1.2pt] at (axis cs: 3 , 4){};  
					\node [draw,circle,fill, blue,inner sep=1.2pt] at (axis cs: 2 , 4){};  
					\node [draw,circle,fill, blue,inner sep=1.2pt] at (axis cs: 2 , 3){};  
					\node [draw,circle,fill, blue,inner sep=1.2pt] at (axis cs: 1 , 3){};  
					\node [draw,circle,fill, blue,inner sep=1.2pt] at (axis cs: 1 , 2){};
					\node [draw,circle,fill, blue,inner sep=1.2pt] at (axis cs: 0 , 2){};  
					\node [draw,circle,fill, blue,inner sep=1.2pt] at (axis cs: 0 , 1){};  
					\node [draw,circle,fill, blue,inner sep=1.2pt] at (axis cs: 0 , 0){};  
					\addplot[red] coordinates {(0,0) (5,5)};
					\addplot[->, thick,  blue] coordinates {(5,5) (5,4)};
					\addplot[->, thick,  blue] coordinates {(5,4) (4,4)};
					\addplot[->, thick,  blue] coordinates {(4,4) (3,4)};
					\addplot[->, thick,  blue] coordinates {(3,4) (2,4)};
					\addplot[->, thick,  blue] coordinates {(2,4) (2,3)};
					\addplot[->, thick,  blue] coordinates {(2,3) (1,3)};
					\addplot[->, thick,  blue] coordinates {(1,3) (1,2)};
					\addplot[->, thick,  blue] coordinates {(1,2) (0,2)};
					\addplot[->, thick,  blue] coordinates {(0,2) (0,1)};
					\addplot[->, thick,  blue] coordinates {(0,1) (0,0)};
					\node [draw,circle,fill, blue,inner sep=2.5pt] at (axis cs: 4 , 4){};  
					\node [draw,circle,fill, blue,inner sep=2.5pt] at (axis cs: 5 , 5){};  
					\node [draw,circle,fill, blue,inner sep=2pt] at (axis cs: 0 , 0){}; 
				\end{axis}
			\end{tikzpicture} 
		}
		\caption{Two bridges of length $10$ from $(5,5)$ to $(0,0)$ with two simultaneous returns.}
		\label{fig:PairBridges}
	\end{figure}

	\[
	\binom{2n}{n}^2 \sim \frac{16^n}{\pi n}, \quad \Leftrightarrow\quad
	P_2(z) \sim \frac{1}{\pi}\log\Big(\frac{1}{1-\frac{z}{\rho}}\Big),
	\]
	for $z$ near $\rho=1/16$. Consequently, 
	\[
	F_2(z)\sim 1-\frac{\pi}{\log\Big(\frac{1}{1-\frac{z}{\rho}}\Big)}.
	\]
\end{example}

The radius of convergence of $P_m(z)$ and $F_m(z)$ is $\rho_m=\frac{1}{4^m}$, by the ratio test. 
Similarly to our results for the (hyper)-cube, the nature of the composition scheme for the $m$-Dyck bridges problem changes beyond after two bridges. 
\begin{theorem}[Composition schemes for $m$-Dyck bridges]
	The composition scheme $\mathcal{P}=\Seq(v\mathcal{F})$ encodes the number of simultaneous returns to the diagonal $x=y$ 
	for $m$ different Dyck bridges. They are critical for one or two Dyck bridges, and subcritical for $m\ge 3$ bridges. In other words, let $\rho=\rho(m)=4^m$ denote the singularity $P_m(z)$, as well as $F_m(z)$ and $\rho_{G}=1$ the singularity of $G(z)=1/(1-z)$. 
	We have
	\begin{equation}
		\lim_{z\to \rho}F_m(z)
		\begin{cases}
			=1, \quad m=1\text{ and } m=2;\\
			<1,\quad m\ge 3.
		\end{cases}
	\end{equation}
\end{theorem}

\begin{remark}
	Similar questions can be studied for other step sets~\cite{BanderierFlajolet2002} with zero drift, like Motzkin bridges, again leading to the same behavior.
\end{remark}
\begin{proof}
	Case $m=1$ corresponds to classical Dyck bridges and the case $m=2$ was presented before in Example~\ref{Ex:twoBridges}, 
	leading to 
	\[
	\lim_{z\to \rho}F_2(z)=1-\lim_{z\to \rho}\frac{1}{P_2(z)}=1=\rho_G.
	\]
	For $m\ge 3$ we observe that the generating function $P(z)$ converges at $z=\rho$ to a positive value, 
	due to the asymptotics of $\binom{2n}{n}^m\sim (4^m)^{n}/(n\pi)^{m/2}$. 
	Thus, as stated
	\[
	\lim_{z\to \rho}F_m(z)=1-\frac{1}{P_m(\rho)}<1=\rho_G.
	\]
\end{proof}

This behavior also reflects itself directly in the distribution of the number of simultaneous returns $X_n$, defined by
\[
\Prb{X_n=k}=\frac{[z^n v^k] P_m(z,v)}{[z^n]P_m(z,1)}=\frac{[z^n]F_m^k(z)}{\binom{2n}{n}^m}.
\]
Following along the lines of our earlier results, the functions in Example~\ref{Ex:twoBridges} can be treated identically to our functions $P(z)$ and $A_1(z)$ discussed before. 
In particular, we note that a refined analysis can also be easily carried out, leading to an extension of the results in~\cite{Li2024}, as in $\sum_{j\ge 0}N_{n,j}=N_n$. 
We state our results for the number of simultaneous returns $N_n$ in an abbreviated way,
as they we simply apply our main theorems. Note that by Theorem~\ref{the:bijDiagonal} $N_n$ also counts the returns to the origin in the diagonal random walk in $\ndZ^2$. 
This corrects an error of~\cite{Li2024}, where the limit law is stated as an ordinary exponential distribution instead of a $\mathrm{Gamma}(2,1)$ distribution, and also extends their result to the refinements $N_{n,j}$.

\begin{corollary}
	The scaled random variable $N_n/(\frac{\log(n)}{\pi})$ converges in distribution 
	to a Gamma distribution with density function $f(x)=xe^{-x}$, $x\ge 0$:
	\[
	\frac{N_n}{\frac{\log(n)}{\pi}}\to \mathrm{Gamma}(2,1).
	\]
	The convergence holds true for all raw moments, and also in terms of a local limit theorem. 
\end{corollary}

\subsection{Colored walks}
\label{sec:coloredwalks}
In the following, we consider colored paths in the cube. Let $r>0$ be an integer.
An $r$-coloured bridge is an $r$-tuple $(B_1,\dots,B_r)$ of (possibly empty) bridges $B_i$. Each bridge itself is a path from a corner an upper corner $(k,k,k)$, $k\in\ndN_0$, to the origin. As a visual representation, we think of them appended one after the other, $B_i$ is colored in color $i$, $1\le i\le r$. Note that not all colors need to be present. See Andrews~\cite{Andrews2007} and~\cite{GhoshDastidarWallner2024,HopkinsOuvry2021} for some combinatorial properties of these walks, and links with multicompositions. We also refer to the recent works~\cite{BaKuWaSt,BaKuWa} for the appearance of colored walks in composition schemes. The bivariate generating function of such $r$-colored bridges is readily obtained using symbolic combinatorics. For a step set consisting of unit steps $\myvec{-1\\0\\0}$, $\myvec{0\\-1\\0}$, $\myvec{0\\0\\-1}$, we get the following generalization of Theorem~\ref{the:DiagToDiag}: The bivariate generating function $A^{[r]}(z,q)$ 
of the number of colored paths with $r$ colors of length $3n$ from $(n,n,n)$ to the origin with exactly $j$ returns to the diagonal
is given by 
\[
A^{[r]}(z,q)=\frac{1}{\big(1-q A_1(z)\big)^r},
\]
with $A_1(z)= 1-\frac{1}{P(z)}$ and $P(z)=\sum_{m\ge 0} \binom{3m}{m,m,m}z^m$~\eqref{P:hypergeom}. Thus, our results apply, with $\alpha=r$, and similarly for Delannoy paths in the cube, as well as the cube walk of Example~\ref{Ex:twoBridges}.

\subsection{Kreweras, lazy Kreweras and diagonal square lattice walks\label{SubSec:squareWalks}}
The lattice paths of interest in the cube are related to walks in the $\ndZ\times\ndZ$ plane with so-called Kreweras steps $\mathcal{S}_K=\{\myvec{-1 \\0}, \myvec{0\\-1},\myvec{1\\1}\}$. 
Kreweras walks in the quarter plane $\ndN_0\times\ndN_0$ recently received a lot of attention~\cite{BM2005,BMM2010,K1965}. 
We map the paths in the cuboid and their coordinates $(x,y,z)\in\ndN_0^3$ to paths in the plane with coordinates $(x-z,y-z)\in\ndZ^2$.
The steps are mapped as follows. Let $\mathcal{S}=\{\myvec{-1\\0\\0},\myvec{0\\-1\\0},\myvec{0\\0\\-1}\}$ denote the unit steps for the random walk in the cuboid. Let $\Phi\colon \mathcal{S}\to \mathcal{S}_K$ denote the bijection between the step sets:
\[
\myvec{-1\\0\\0}\xmapsto{\Phi} \myvec{-1\\0},\quad
\myvec{0\\-1\\0}\xmapsto{\Phi} \myvec{0\\-1},\quad
\myvec{0\\0\\-1}\xmapsto{\Phi} \myvec{1\\1}.
\]
This is readily extended to $\Phi\colon \ndN_0^3\to\ndZ^2$: for any path $p$ in $\ndN_0^3$ by apply the restriction to the individual steps. 
One may also look at an arbitrary starting point $(n_1,n_2,n_3)$ in the cuboid. We collect our findings in the following theorem.

\begin{theorem}
	\label{the:Kreweras}
	The simple random walks from $(n,n,n)$ to the origin are in bijection $\Phi$ with Kreweras walks in the plane of length $3n$, starting and ending at the origin. In particular, the number of times that the paths return to the diagonal $x=y=z$ equals the number of returns to the origin in the corresponding Kreweras walk. 
\end{theorem}

\begin{figure}[!htb]
	\resizebox{6cm}{6cm}
	{
		\begin{tikzpicture}
			\begin{axis}[
				grid=major, xlabel=$x$,ylabel=$y$,zlabel=$z$,
				xtick={0,...,6},  ytick={0,...,5}, ztick={0,...,4},
				] 
				\node [draw,circle,fill, blue,inner sep=2pt] at (axis cs: 0 , 0,0){}; 
				\node [draw,circle,fill, blue,inner sep=2pt] at (axis cs: 6 , 5,4){}; 
				\addplot3[mark=square*, red] coordinates {(0,0,0) (1,1,1) (2,2,2)(3,3,3)(4,4,4)};
				\addplot3[mark=*, blue] coordinates {(6,5,4)(6,4,4) (5,4,4) (4,4,4)(3,4,4)(2,4,4)(2,4,3)(2,4,2)(2,3,2)(2,2,2)(1,2,2)(0,2,2)(0,1,2)(0,1,1)(0,0,1)(0,0,0)};
			\end{axis}
		\end{tikzpicture} 
	}
	\quad
	\resizebox{6cm}{6cm}
	{
		\begin{tikzpicture}
			\begin{axis}[
				axis lines*=left,
				clip=false, no markers,
				xlabel=$x$,ylabel=$y$,  axis lines=middle,grid=major,
				xtick={-4,...,4},  ytick={-4,...,4},
				xmin=-2, xmax=2,ymin=-2, ymax=2
				] 
				\node [draw,circle,fill, blue,inner sep=2pt] at (axis cs: 2 , 1){}; 
				\node [draw,circle,fill, blue,inner sep=1.2pt] at (axis cs: 2 , 0){}; 
				\node [draw,circle,fill, blue,inner sep=1.2pt] at (axis cs: 1 , 0){};   
				\node [draw,circle,fill, blue,inner sep=1.2pt] at (axis cs: 0 , 0){};   
				\node [draw,circle,fill, blue,inner sep=1.2pt] at (axis cs: -1 , 0){};   
				\node [draw,circle,fill, blue,inner sep=1.2pt] at (axis cs: -2 , 0){};  
				\node [draw,circle,fill, blue,inner sep=1.2pt] at (axis cs: -1 , 1){};   
				\node [draw,circle,fill, blue,inner sep=1.2pt] at (axis cs: 0 , 2){};   
				\node [draw,circle,fill, blue,inner sep=1.2pt] at (axis cs: 0 , 1){}; 
				\node [draw,circle,fill, blue,inner sep=1.2pt] at (axis cs: 0 , 0){};  
				\node [draw,circle,fill, blue,inner sep=1.2pt] at (axis cs: -2 , -1){};  
				\node [draw,circle,fill, blue,inner sep=1.2pt] at (axis cs: -1 , -1){};    
				\addplot[->, thick,  blue] coordinates {(2,1) (2,0)};
				\addplot[->, thick,  blue] coordinates {(2,0) (1,0)}; 
				\addplot[->, thick,  blue] coordinates {(1,0) (0,0)}; 
				\addplot[->, thick,  blue] coordinates {(0,0) (-1,0)}; 
				\addplot[->, thick,  blue] coordinates {(-1,0) (-2,0)}; 
				\addplot[->, thick,  blue] coordinates {(-2,0) (-1,1)};
				\addplot[->, thick,  blue] coordinates {(-1,1) (0,2)};
				\addplot[->, thick,  blue] coordinates {(0,2) (0,1)}; 
				\addplot[->, thick,  blue] coordinates {(0,1) (0,0)}; 
				\addplot[->, thick,  blue] coordinates {(0,0) (-1,0)}; 
				\addplot[->, thick,  blue] coordinates {(-1,0) (-2,0)}; 
				\addplot[->, thick,  blue] coordinates {(-2,0) (-2,-1)};
				\addplot[->, thick,  blue] coordinates {(-2,-1) (-1,0)};
				\addplot[->, thick,  blue] coordinates {(-1,0) (-1,-1)};
				\addplot[->, thick,  blue] coordinates {(-1,-1) (0,0)};
				\node [draw,circle,fill, blue,inner sep=2pt] at (axis cs: 0 , 0){}; 
			\end{axis}
		\end{tikzpicture} 
	}
	\caption{A sample path $\omega$ of length $15$ from $(6,5,4)$ to $(0,0,0)$
		and the corresponding walk with Kreweras steps from $(2,1)$ to $(0,0)$.}
\end{figure}

At the end of this example we note a bijection $\Psi$ between Delannoy paths, starting at $(n,n,n)$ and ending at the origin,
to what we call lazy Kreweras-type walks in $\ndZ^{2}$, starting and ending at the origin, with step set 
\[
\mathcal{S}_{LK}=\{\myvec{-1\\0},\myvec{0\\-1},\myvec{1\\1},\myvec{0\\0}\}.
\]
We have
\[
\myvec{-1\\0\\0}\xmapsto{\Psi} \myvec{-1\\0},\quad\myvec{0\\-1\\0}\xmapsto{\Psi} \myvec{0\\-1},\quad
\myvec{0\\0\\-1}\xmapsto{\Psi} \myvec{1\\1},\quad \myvec{-1\\-1\\-1}\xmapsto{\Psi}\myvec{0\\0},
\]
and the bijection extends to paths by application $\Psi$ to each edge on the particular path. The bijection naturally extends higher dimensions. Delannoy paths in the hypercube in $\ndN_0^m$, $m\ge 3$,
are defined by the step set 
\[
\mathcal{S}_{D}=\{-\vec{e}_{1,m},\dots,-\vec{e}_{m,m},-\sum_{k=1}^{m}\vec{e}_{k,m}\}
\]
and lazy Kreweras walks in $\ndZ^{m-1}$, starting and ending at the origin, with step set 
\[
\mathcal{S}_{LK}=\{-e_{1,m-1},\dots,-e_{m-1,m-1},\sum_{k=1}^{m-1}e_{k,m-1},\vec{0}\}.
\]
Here, $\vec{e}_{k,j}$ denotes the $k$th unit vector in $\ndZ^j$, $j\ge 1$. Then, the restricted map $\Psi\colon \mathcal{S}_{D}\to \mathcal{S}_{LK}$ is defined by
\[
-\vec{e}_{k,m}\xmapsto{\Psi}\vec{e}_{k,m-1},\quad\text{for }\,1\le k\le m-1,
\] 
and
\[
-\vec{e}_{m,m}\xmapsto{\Psi}\sum_{k=1}^{m-1}e_{k,m-1},\quad
-\sum_{k=1}^{m}\vec{e}_{k,m}\xmapsto{\Psi}\vec{0}.
\]

\begin{figure}[!htb]
	\centering
	\resizebox{5cm}{5cm}
	{
		\begin{tikzpicture}
			\begin{axis}[
				axis lines*=left,
				clip=false, no markers,
				xlabel=$x$,ylabel=$y$,  axis lines=middle,grid=major,
				xtick={-4,...,4},  ytick={-4,...,4},
				xmin=-2, xmax=2,ymin=-2, ymax=2
				] 
				\node [draw,circle,fill, blue,inner sep=1.2pt] at (axis cs: -1 , 1){}; 
				\node [draw,circle,fill, blue,inner sep=1.2pt] at (axis cs: -2 , 2){};  
				\node [draw,circle,fill, blue,inner sep=1.2pt] at (axis cs: 0 , 2){};  
				\node [draw,circle,fill, blue,inner sep=1.2pt] at (axis cs: 1 , 0){};  
				\node [draw,circle,fill, blue,inner sep=1.2pt] at (axis cs: 1 , 1){};  
				\node [draw,circle,fill, blue,inner sep=1.2pt] at (axis cs: 1 , -1){};  
				\addplot[->, thick,  blue] coordinates {(0,0) (-1,1)};
				\addplot[->, thick,  blue] coordinates {(-1,1) (-2,2)};
				\addplot[->, thick,  blue] coordinates {(-2,2) (-1,1)};
				\addplot[->, thick,  blue] coordinates {(-1,1) (0,2)};
				\addplot[->, thick,  blue] coordinates {(0,2) (-1,1)};
				\addplot[->, thick,  blue] coordinates {(-1,1) (0,2)};
				\addplot[->, thick,  blue] coordinates {(0,2) (1,1)};
				\addplot[->, thick,  blue] coordinates {(1,1) (0,0)};
				\addplot[->, thick,  blue] coordinates {(0,0) (1,-1)};
				\addplot[->, thick,  blue] coordinates {(1,-1) (0,0)};
				\node [draw,circle,fill, blue,inner sep=2pt] at (axis cs: 0 , 0){}; 
			\end{axis}
		\end{tikzpicture} 
	}
	\caption{The diagonal walk in $\ndR^2$ of length ten, corresponding to the two bridges in Figure~\ref{fig:PairBridges}: $\binom00\to\binom{-1}1\to\binom{-2}{2}\to\binom{-1}{1}\to\binom{0}{2}\to\binom{-1}{1}\to\binom{0}{2}\to\binom{1}{1}\to\binom{0}{0}
		\to\binom{1}{-1}\to\binom{0}{0}$}
\end{figure}
Before we turn to enumerative results
we note that the collection of $m\ge 2$ Dyck bridges also has a corresponding model of a random walk in $\ndZ^m$. 

\begin{lemma}[Collections of Dyck bridges and diagonal walks]
	\label{the:bijDiagonal}
	The sample paths of $m\ge 2$ Dyck bridges of length $2n$ are in bijection with a random walk of length $2n$ in $\ndZ^m$,
	starting and ending at $(0,\dots,0)$, with step set 
	\[
	\mathcal{S}_m=\{(\pm 1,\pm 1, \dots, \pm 1)\}
	\]
	of size $2^m$. Moreover, the simultaneous returns of the $m$ Dyck bridges equal the number of returns to the origin in $\ndZ^{m}$.
	In particular, a pair of Dyck bridges of length $2n$ are in bijection with so-called diagonal random walks in $\ndZ^2$ of length $2n$, 
	with steps set
	\[
	\mathcal{S}_2=\{(1,1),(-1,1),(1,-1),(-1,-1)\}.
	\]
	Furthermore, the pair of bridges is also in bijection with the walk, starting at $(n,n,n)$, ending at the origin and steps
	\[
	\myvec{-1\\0\\0}, \myvec{0\\-1\\0},\myvec{-1\\0\\-1},
	\myvec{0\\-1\\-1}.
	\]
	
\end{lemma}
\begin{proof}
	This bijection is a folklore result, at least for $m=2$, see~\cite{BOR2019,BaKuWaSt}. We reverse the direction of the Dyck bridges such that they end at $(n,n)$. Then we rotate the coordinate system by 45 degrees, leading $m$ directed lattice walks with steps $(1,1)$ and $(1,-1)$. Then, we map the $y$-coordinate of the $k$th bridge, with $1\le k\le m$, to the $k$th coordinate of the random walk in $\ndZ^m$. This leads to the desired bijection. 
	In particular, we note that a return to zero of the $k$th bridge corresponds to a contact with $x_k=0$. Simultaneous returns thus imply a return to the origin in $\ndZ^m$. 
	Concerning the bijection stated at the end, we observe that $k$ steps $\myvec{-1\\0\\0}$ 
	imply that $n-k$ steps are of type $\myvec{-1\\0\\-1}$. Thus, we require also 
	$k$ steps $\myvec{0\\-1\\-1}$  and $n-k$ steps $\myvec{0\\-1\\0}$, leading to a suitable bijection.
\end{proof}

\subsection{Sampling without replacement urn}
\label{sec:urn}
The sampling without replacement urn is a basic and fundamental urn model. One draws balls at random one after the other, observes the color of the drawn ball 
and removes it. One may study different questions regarding this urn, we refer the interested reader to~\cite{FlaDuPuy2006,HKP2007} The ball transition matrix of the sampling without replacement urn with three colors is given by
\[
\left(
\begin{matrix}
	-1 & 0 & 0\\
	0 & -1 & 0\\
	0 & 0 & -1
\end{matrix}
\right)
\]
We note in passing that, in the general case, the replacement matrix is the negative identity matrix. The sample paths of the urn model take place in $\ndN_0^3$. We provide the following connection between lattice paths, the diagonal, and the sampling urn. 

\begin{proposition}
	The random variable $N_n$, counting the number of returns to the diagonal $x=y=z$ in the cube counts the number of times there are equally many balls in the sampling urn process, starting with $n=n_1=n_2=n_3$ balls of each type.    
\end{proposition}
Thus, our previous results provide a limit law for the number of equalities for the sampling without replacement urn. 
\begin{proof}
	We look at the sample paths of the Sampling without replacement urn, starting with $n_k$ balls of each type, $1\le k\le 3$, and the return to the diagonal $d\colon x=y=z$. By definition of the urn model, the transfer probabilities are given by 
	\begin{equation}
		\begin{split}
			\label{eqn:transfer}
			\Prb{(n_1,n_2,n_3)\to(n_1-1,n_2,n_3)}&=\frac{n_1}{n_1+n_2+n_3},\\
			\Prb{(n_1,n_2,n_3)\to(n_1,n_2-1,n_3)}&=\frac{n_2}{n_1+n_2+n_3},\\
			\Prb{(n_1,n_2,n_3)\to(n_1,n_2,n_3-1)}&=\frac{n_3}{n_1+n_2+n_3},
		\end{split}
	\end{equation}
	for $n_1,n_2,n_3\ge 0$ with at least one nonzero value.
	Thus, we easily observe that for every sample path $\sigma\in\ndN_0^3$ equation~\ref{eqn:transfer} implies that
	\[
	\binom{n_1+n_2+n_3}{n_1,n_2,n_3}\cdot \mathbb{P}_{\mathcal{S}}\{\sigma\}=1,
	\]
	where $\mathbb{P}_{\mathcal{S}}$ denotes the probability measure induced by the process. Thus, the draw after draw evolution of the urn can be modeled by the uniform distribution in the cube. 
\end{proof}

\subsection{Card Guessing games}
\label{sec:cards}
Card guessing games are of great interest and have been considered in the literature in many articles, starting with the works of Diaconis~\cite{Diaconis1978} 
and Diaconis and Graham~\cite{DiaconisGraham1981}; we refer the reader to~\cite{PK2023} for more pointers to the literature, in particular for
applications to real world problems. The card guessing game of our main focus can be described as follows. A deck of a total of $N$ cards is shuffled, and then the guesser in provided with the total number cards $N$, as well as the individual numbers of say hearts, diamonds, clubs and spades. After each guess, the person guessing the cards is shown the drawn card, which is then removed from the deck. This process is continued until no more cards are left. Assuming the guesser tries to maximize the number of correct guesses, one is interested in the total number of correct guesses. Of course, this card guessing procedure can be considered with an arbitrary number $m\ge 2$ of different types of cards. The simplest case of $m=2$ different types of cards, say, colors red (hearts and diamonds) and black (clubs or spades), is now very well understood, leading to connections to urn models, Dyck paths, hitting times in Brownian bridges~\cite{DiaconisGraham1981}, as well as new phenomena in Analytic Combinatorics; see~\cite{PK2023} for details. 

\smallskip

In the general setting of $m$ different types, a main interest has been in deriving asymptotics for the expectation of the total number of correct guesses~\cite{DiaconisGraham1981,HeOttolini2021,OttoliniSteiner2022}. Here, two different settings have been treated in the literature so far. In the first setting, one starts with a fixed number of cards $m$ of different types and equally many cards $n=n_k$ of each type $k$, $1\le k\le m$, as $n$ tends to infinity. Second, one studies the case of $m\to\infty$ and the individual numbers of cards $n_k\in\ndN$ all equal and fixed, or bounded. It highly desirable to shed more light on the fundamental result of Diaconis and Graham for the expected value in the first setting, mentioned before.
\begin{theorem}[Diaconis-Graham~\cite{DiaconisGraham1981}]
	\label{the:DG}
	The expected value $\Exb{C_{n,\dots,n}}$ of the number of correct guesses satisfies for fixed $m\in\ndN$, with $m\ge 2$ fixed, and $n\to\infty$ the expansion
	\[
	\Exb{C_{n,\dots,n}}=n + \frac{\pi}2 \cdot M \cdot \sqrt{n} + o(\sqrt{n}),
	\]
	where $M=\Exb{\max\{N_1,\dots,N_m\}}$ and the random variables $N_k$ are iid standard normal distributed.
\end{theorem}
Except for the case $m=2$, it is an open problem to derive the limit law of $C_{n,\dots,n}$, or even raw moments beyond the expectation. We look into the case of $m=3$ and use a decomposition of $C_{n,n,n}$ into different random variables. Such decompositions have turned out to be very fruitful in the two-color case. We use a link between or previous results for the returns to the diagonal in the cube and the card guessing game to obtain more insight on $C_{n,n,n}$.
The idea is to decompose the correct guesses according to two aspects: the number of types of cards being maximal in the deck when a correct guess occurs and the number of different types still present in the deck. 
\[
C_{n,n,n}=\sum_{k=1}^3\sum_{\ell=1}^{k}C_{n,n,n}^{(k,\ell)},
\]
where $C_{n,n,n}^{(k,\ell)}$ denotes the number of correct guesses with $\ell$ cards being maximal in the deck, while $k$ different types are still present.
We refer the reader to~\cite{KuPa2025Three} for more details of this decomposition, as well as to~\cite{PK2023} for a similar decomposition for two types of cards. Several simplifications occur~\cite{KuPa2025Three}, and the main interest is into the random variable 
$C_{n,n,n}^{(3,3)}$, as well as $C_{n,n,n}^{(3,2)}$ and $C_{n,n,n}^{(2,2)}$. Our results will allow us to give a complete description of $C_{n,n,n}^{(3,3)}$.
\begin{proposition}[~\cite{KuPa2025Three}]
	The number of pure luck guesses $C_{n,n,n}^{(3,3)}$ in a card guessing game with $n$ cards of each of the three types is related to the number $N_n$ of returns to the diagonal in a simple random walk
	in the cube:
	\begin{equation*}
		\label{eqn:CoinToss1}
		C_{n,n,n}^{(3,3)} = \bin \big(N_n,\textstyle{\frac{1}{3}}\big).
	\end{equation*}   
\end{proposition}
\begin{corollary}
	\label{Cor:Guess}
	The random variable $C_{n,n,n}^{(3,3)}$ has a Gamma limit law:
	to an Erlang distribution with density function $f(x)=xe^{-x}$, $x\ge 0$:
	\[
	\sqrt{3} C_{n,n,n}^{(3,3)}/(\frac{1}{2\pi}\log(n))\to\text{Gamma}(2,1).
	\]
\end{corollary}
\begin{proof}[Proof of Corollary~\ref{Cor:Guess}]
	We use a well-known classical property of the binomial distribution $R= \bin(n,p)$:
	\[
	\Exb{\fallfak{R}{s}}=\fallfak{n}{s}p^s, \quad s\ge 0.
	\]
	Thus, we obtain for the conditional expectation 
	\[
	\Exb{\fallfak{\big(C_{n,n,n}^{(3,3)}\big)}{s}\mid N_n} = \fallfak{N_n}{s}\cdot \frac{1}{3^s}.
	\]
	Consequently, We obtain the factorial moments of $C_{n,n,n}^{(3,3)}$ by the tower rule of total expectation.
	We get the desired expansion for the raw moments
	\begin{equation*}
		\Exb{(N_n)^s} = \sum_{k=0}^s \Stir{s}{k} \Exb{\fallfak{(N_n)}k} \sim \frac{1}{3^s}\cdot (s+1)!\Big(\frac{\sqrt{3}}{2\pi}\Big)^{s}\log^{s}(n).
	\end{equation*}
	Carleman's criterion~\cite[pp.~189--220]{Carleman23} for the Stieltjes moment problem, support $[0,\infty)$, states
	that if 
	\begin{equation}
		\sum_{s=0}^{\infty}\mu_s^{-1/(2s)}=+\infty,
		\label{eq:Carleman}
	\end{equation}
	then the moment sequence $(\mu_s)_{s\ge 1}$ determines a unique distribution. Furthermore, this implies that if there exists a constant $C>0$ such that
	\[
	\mu_s\le C^s(2s)!\quad\text{for } s\in\ndN,
	\]
	then Carleman's criterion is satisfied. This is obviously true for the moments of the Erlang distribution,
	such that the divergence in Carleman's criterion is satisfied. Thus, by the Fr\'echet--Shohat theorem~\cite{FrSh1931}, 
	we obtain the weak convergence of the scaled random variable to a Gamma-distributed random variable.
\end{proof}

\subsection{Outlook}
\label{sec:outlook}
At the end of this work we outline how to apply our main results to the analysis of space contacts for other families of lattice paths in the cube, starting at the upper corner $(n,n,n)$ and ending at the origin. In order to automate the derivation of the overall asymptotics of the lattice paths
from $(n,n,n)$ to the origin with steps in $\mathcal{S}$, we reverse the direction of the steps and study the generating function
\[
F(t;x,y,z)=\sum_{n\ge 0}\big(P(x,y,z)\big)^n t^n=\frac1{1-tP(x,y,z)},
\]
where $P(x,y,z)=P_{-\mathcal{S}}(x,y,z)$ denotes the step polynomial associated to the step set $-\mathcal{S}$.
The number of walks from the upper corner $(n,n,n)$ to the origin is then encoded by the diagonal: $[x^ny^nz^n]F(1;x,y,z)$. The derivation of such diagonals and their asymptotics can be done using Analytic Combinatorics in several variables~\cite{MelczerSalvy2021,ACSV}, see also~\cite{Bostan_2015}.
Then, assuming that the walks are directed, we can simply use the arch decomposition, 
leading to a combinatorial sequence construction and $V(z)=1/(1-z)$. Then, the asymptotics of the diagonal of $F(1;x,y,z)$ are translated to the asymptotics of $W(z)$, which allows to check our standing assumptions and invoke our main results. 

\begin{table}[!htb]
	\centering
	\begin{tabular}{|c|c|}
		\hline
		Step set $-\mathcal{S}$ & $F(1;x,y,z)$ \\
		\hline
		$\{\vec{e}_1,\vec{e}_2,\vec{e}_3\}$ & $1/(1-x-y-z)$\\
		\hline
		$\{\vec{e}_1,\vec{e}_2,\vec{e}_3,\vec{e}_1+\vec{e}_2+\vec{e}_3\}$ & $1/(1-x-y-z-xyz)$\\
		\hline
		$\{\vec{e}_1+\vec{e}_2,\vec{e}_1+\vec{e}_3,\vec{e}_2+\vec{e}_3,\vec{e}_1+\vec{e}_2+\vec{e}_3\}$ & $1/(1-xy-xz-xy-xyz)$\\
		\hline
		$\{\vec{e}_1,\vec{e}_2,\vec{e}_1+\vec{e}_3,\vec{e}_2+\vec{e}_3\}$ & $1/(1-(x+y)(z+1))$\\
		\hline
	\end{tabular}
	\caption{Reversed stepset $-\mathcal{S}$ and overall generating function $F(1;x,y,z)$ for walks discussed in this work}
	\label{tab:placeholder}
\end{table}
Finally, we note that one can readily study different step sets, colored walks, or even walks with colored steps.

\section*{Acknowledgments}
The authors thank Alois Panholzer for several discussions about combinatorial applications of our results and his interest in this work. The second author also thanks Cyril Banderier for interesting discussions concerning the new component weights treated in this article and their connection to families of lattice paths. 

\bibliographystyle{abbrv}
\bibliography{gibbs}

\end{document}